\documentclass[letterpaper, reqno,11pt]{article}
\usepackage[margin=1.0in]{geometry}
\usepackage{geometry}                % See geometry.pdf to learn the layout options. There are lots.
\usepackage{graphicx}
\usepackage{amssymb}
\usepackage{epstopdf}
\usepackage{bm}

\usepackage{amsmath}
\usepackage{amsthm} % only needed if we are using article rather than amsart

\newtheorem{prop}{Proposition}[section]
\newtheorem{lemma}[prop]{Lemma}
\newtheorem{cor}[prop]{Corollary}
\newtheorem{thm}[prop]{Theorem}
\newtheorem{defn}[prop]{Definition}
\newtheorem{rem}[prop]{Remark}

\newtheorem*{quantACCPropProp}{Proposition \ref{quantACCProp}}

\usepackage{color} % this is to make the TODOs in red and the comments to Josh/Larry in green/blue
\usepackage{enumitem} % to make references to enumerated lists display in roman

\newcommand{\mc}{\mathcal{C}}

\newcommand{\CC}{\mathbb{C}}

\newcommand{\RR}{\mathbb{R}}
\newcommand{\NN}{\mathbb{N}}
\newcommand{\Poly}{\operatorname{Poly}}
\newcommand{\Deg}{\operatorname{deg}}
\newcommand{\FP}{K \mathbf{P}}
\newcommand{\BZ}{Z}
\newcommand{\codim}{\operatorname{codim}}

\newcommand{\pts}{\mathcal{P}}
\newcommand*\bell{\ensuremath{\boldsymbol\ell}}
\newcommand{\smooth}{\operatorname{reg}}
\newcommand{\mult}{\operatorname{mult}}
\newcommand{\complexity}{\operatorname{complexity}}
\newcommand{\FZT}{F_{Z(T)}}
\newcommand{\Flec}{\operatorname{Flec}}
\newcommand{\chara}{\operatorname{char}}

\newcommand{\finiteSetCurves}{\mathcal{L}}
\newcommand{\chowVarietyDegD}{\mathcal{C}_{3,D}}
\newcommand{\tildeChowVarietyDegD}{\mathcal{C}^*_{3,D}}

\newcommand{\chowVarietyDegExactlyD}{\tilde{\mathcal{C}}_{3,D}}
\newcommand{\tildeChowVarietyDegExactlyD}{\tilde{\mathcal{C}}^*_{3,D}}

\newcommand{\ProjChowVarietyDegD}{\hat{\mathcal{C}}_{3,D}}
\newcommand{\tildeProjChowVarietyDegD}{\hat{\mathcal{C}}^*_{3,D}}

\newcommand{\defeq}{:=}

\newtheorem*{mainStructureThmThm}{Theorem \ref{mainStructureThm}}
\newtheorem*{implicationsOfDoublyRuledProp}{Proposition \ref{implicationsOfDoublyRuled}}
\newtheorem*{degreeReductionProp}{Proposition \ref{degreeReduction}}

\begin{document}
\title{Algebraic curves, rich points, and doubly-ruled surfaces}

\author{
Larry Guth\thanks{Massachusetts Institute of Technology, Cambridge, MA,  {\sl lguth@math.mit.edu}.}
\and
Joshua Zahl\thanks{Massachusetts Institute of Technology, Cambridge, MA,  {\sl jzahl@math.mit.edu}.}
}
\date{\today}

\maketitle
\begin{abstract}
We study the structure of collections of algebraic curves in three dimensions that have many curve-curve incidences. In particular, let $k$ be a field and let $\mathcal{L}$ be a collection of $n$ space curves in $k^3$, with $n<\!\!<(\operatorname{char}(k))^2$ or $\operatorname{char}(k)=0$. Then either A) there are at most $O(n^{3/2})$ points in $k^3$ hit by at least two curves, or B) at least $\Omega(n^{1/2})$ curves from $\mathcal{L}$ must lie on a bounded-degree surface, and many of the curves must form two ``rulings'' of this surface.

We also develop several new tools including a generalization of the classical flecnode polynomial of Salmon and new algebraic techniques for dealing with this generalized flecnode polynomial.
\end{abstract}

\section{Introduction}\label{introSection}
If $\mathcal{L}$ is a collection of $n$ lines in the plane, then there can be as many as $\binom{n}{2}$ points that are incident to at least two lines. If instead $\mathcal{L}$ is a collection of $n$ lines in $\RR^3$, there can still be $\binom{n}{2}$ points that are incident to two or more lines; for example, this occurs if we choose $n$ lines that all lie in a common plane in $\RR^3$, with no two lines parallel. A similar number of incidences can be achieved if the lines are arranged into the rulings of a regulus. This leads to the question: if we have a collection of lines in $\RR^3$ such that many points are incident to two or more lines, must many of these lines lie in a common plane or regulus?

This question has been partially answered by the following theorem of Guth and Katz:
\begin{thm}[Guth-Katz \cite{GuthKatz}]\label{linesThm}
Let $\mathcal{L}$ be a collection of $n$ lines in $\RR^3$. Let $A\geq 100n^{1/2}$ and suppose that there are $\geq 100An$ points $p\in\RR^3$ that are incident to at least two points of $\mathcal{L}$. Then there exists a plane or regulus $Z\subset\RR^3$ that contains at least $A$ lines from $\mathcal{L}$.
\end{thm}

In this paper, we will show that a similar result holds for more general curves in $k^3$:
\begin{thm}\label{degDCurvesThm}
Fix $D>0$. Then there are constants $c_1,C_1,C_2$ so that the following holds. Let $k$ be a field and let $\finiteSetCurves$ be a collection of $n$ irreducible curves in $k^3$ of degree at most $D$. Suppose that $\operatorname{char}(k)=0$ or $n\leq c_1 (\operatorname{char}(k))^2.$ Then for each $A\geq C_1n^{1/2}$, either there are at most $C_2An$ points in $k^3$ incident to two or more curves from $\finiteSetCurves$, or there is an irreducible surface $Z$ of degree $\leq 100D^2$ that contains at least $A$ curves from $\finiteSetCurves$.
\end{thm}

In fact, more is true. If all the curves in $\finiteSetCurves$ lie in a particular family of curves, then the surface $Z$ described above is ``doubly ruled'' by curves from this family. For example, if all the curves in $\finiteSetCurves$ are circles, then $Z$ is doubly ruled by circles; this means that there are (at least) two circles passing through a generic point of $Z$. In order to state this more precisely, we will first need to say what it means for a finite set $\finiteSetCurves$ of curves to lie in a certain family of curves. We will do this in Section \ref{chowVarietySec} and state a precise version of the stronger theorem at the end of that section.

\subsection{Proof sketch and main ideas}
As in the proof of Theorem \ref{linesThm}, we will first find an algebraic surface $Z$ that contains the curves from $\finiteSetCurves$. A ``degree reduction'' argument will allow us to find a surface of degree roughly $n/A <\!\!< n^{1/2}$ with this property. In \cite{GuthKatz}, the first author and Katz consider the flecnode polynomial, which describes the local geometry of an algebraic surface $Z\subset\RR^3$. This polynomial is adapted to describing surfaces that are ruled by lines. In the present work, we construct a generalization of this polynomial that lets us measure many geometric properties of a variety $Z\subset K^3$. In particular, we will find a generalized flecnode polynomial that tells us when a surface is doubly ruled by curves from some specified family.

We will also explore a phenomena that we call ``sufficiently tangent implies trapped.'' If a curve is ``sufficiently tangent'' to a surface, then the curve must be contained in that surface. More precisely, if we fix a number $D\geq 1$, then for any surface $Z\subset K^3$ we can find a Zariski open subset $O\subset Z$ so that any irreducible curve of degree $\leq D$ that is tangent to $Z$ at a point $z\in O$ to order $\Omega_{D}(1)$ must be contained in $Z$; the key point is that the necessary order of tangency is independent of the degree of $Z$. This will be discussed further in Section \ref{suffTangTrappedSec}. The result is analogous to the Cayley-Monge-Salmon theorem, which says that if there is a line tangent to order at least three at every smooth point of a variety, then this variety must be ruled by lines (see \cite{Salmon}, or \cite{Landsberg,Katz} for a more modern treatment).

Finally, we will prove several structural statements about surfaces in $K^3$ that are doubly ruled by curves. A classical argument shows that any algebraic surface in $K^3$ that is doubly ruled by lines must be of degree one or two. We will prove an analogous statement that any surface that is doubly ruled curves of degree $\leq D$. This will be done in Section \ref{propDoubRuledSurfSec}.
\subsection{Previous work}
In \cite{GuthKatz}, the first author and Katz show that given a collection of $n$ lines in $\RR^3$ with at most $n^{1/2}$ lines lying in a common plane or regulus, there are at most $n^{3/2}$ points in $\RR^3$ that are incident to two or more lines. This result was a major component of their proof of the Erd\H{o}s distinct distance problem in the plane.

In \cite{Kollar}, Koll\'ar extends this result to arbitrary fields (provided the characteristic is $0$ or larger than $\sqrt{n}$). Koll\'ar's techniques differ from the ones in the present paper. In particular, Koll\'ar traps the lines in a complete intersection of two surfaces, and then uses tools from algebraic geometry to control the degree of this complete intersection variety. In \cite{SS}, Sharir and Solomon consider a similar point-line incidence problem, and they provide a new proof of some of the incidence results in \cite{GuthKatz}.

In \cite{Guth}, the first author showed that given a collection of $n$ lines in $\RR^3$ and any $\epsilon>0$, there are $O_\epsilon(n^{3/2+\epsilon})$ points that are incident to two or more lines, provided at most $n^{1/2-\epsilon}$ lines lie in any algebraic surface of degree $O_\epsilon(1)$. This proof also applies to bounded-degree curves in place of lines. However, the proof is limited to the field $k=\RR$, and unlike the present work, it does not say anything about the structure of the surfaces containing many curves.
\subsection{Thanks}
The first author was supported by a Sloan fellowship and a Simons Investigator award.  The second author was supported by a NSF mathematical sciences postdoctoral fellowship.
\section{Constructible sets}\label{constSetSec}

\begin{defn}\label{defn1Const}
Let $K$ be an algebraically closed field. A constructible set $Y \subset K^N$ is a finite boolean combination of algebraic sets.  This means the following:

\begin{itemize}
\item There is a finite list of polynomials $f_j: K^N \rightarrow K$, $j =1, ..., J(Y)$.
\item Define $v(f_j)$ to be 0 if $f_j = 0$ and $1$ if $f_j \not= 0$.  The vector $v(f_j) = (v (f_1) , ..., v (f_j), ...)$ gives a map from $K^N$ to the boolean cube $\{ 0, 1 \}^{J(Y)}$.
\item There is a subset $B_Y \subset \{0, 1 \}^J$ so that $x \in Y$ if and only if $v(f_j (x) )\in B_Y$.
\end{itemize}
\end{defn}
The constructible sets form a Boolean algebra. This means that finite unions and intersections of constructible sets are constructible, and the compliment of a constructible set is constructible.

\begin{defn}
If $Y$ is a constructible set, we define the complexity of $Y$ to be $\min(\deg f_1+\ldots+\deg f_{J(Y)})$, where the minimum is taken over all representations of $Y$, as described in Definition \ref{defn1Const}.
\end{defn}
\begin{rem}
This definition of complexity is not standard. However, we will only be interested in constructible sets of ``bounded complexity,'' and the complexity of a constructible set will never appear in any bounds in a quantitative way. Thus any other reasonable definition of complexity would work equally well.
\end{rem}

The crucial result about constructible sets is Chevalley's theorem.
\begin{thm}[Chevalley; see  \cite{harris}, Theorem 3.16 or \cite{matsumura}, Chapter 2.6, Theorem 6]\label{ChevalleyTheorem}
Let $X\subset K^{M+N}$ be a constructible set of complexity $\leq C$. Let $\pi$ be the projection from $K^{M} \times K^N$ to $K^M$. Then $\pi(X)$ is a constructible set in $K^M$ of complexity $O_C(1)$.
\end{thm}

Here is an illustrative example.  Suppose that $X$ is the zero-set of $z_1 z_2 - 1$ in $\CC^2 = \CC \times \CC$.  The set $X$ is an algebraic set, and it is certainly constructible.  When we project $X$ to the first factor, we get $ \CC \setminus \{ 0 \}$.  This projection is not an algebraic set, but it is constructible.

We will also need to define constructible sets in projective space.
\begin{defn}
A (projective) constructible set $Y \subset \FP^N$ is a finite boolean combination of projective algebraic sets.  This means the following:

\begin{itemize}
\item There is a finite list of homogeneous polynomials $f_j: K^{N+1} \rightarrow K$, $j =1, ..., J(Y)$.
\item Define $v(f_j)$ to be 0 if $f_j = 0$ and $1$ if $f_j \not= 0$.  The vector $v(f_j) = (v (f_1) , ..., v (f_j), ...)$ gives a map from $\FP^N$ to the boolean cube $\{ 0, 1 \}^{J(Y)}$.
\item There is a subset $B_Y \subset \{0, 1 \}^J$ so that $x \in Y$ if and only if $v(f_j (x) )\in B_Y$.
\end{itemize}
\end{defn}
\section{The Chow variety of algebraic curves}\label{chowVarietySec}

The set of all degree $D$ curves in 3-dimensional space turns out to be algebraic object in its own right, and this algebraic structure will play a role in our arguments.  This object is called a Chow variety.

The (projective) Chow variety $\ProjChowVarietyDegD$ of degree $D$ irreducible curves in $\FP^3$ is a quasi-projective variety, which is contained in some projective space $\FP^N$, $N=N(D)$. Each irreducible degree $D$ curve $\gamma\in\FP^3$ corresponds to a unique point in the Chow variety, called the ``chow point'' of $\gamma$.  Here is a fuller description of the properties of the Chow variety.

\begin{prop}[Properties of the (projective) Chow variety]\label{propertiesOfProjChow}
Let $D\geq 1$ Then there is a number $N=N(D)$ and sets $\ProjChowVarietyDegD \subset \FP^N$, $\tildeProjChowVarietyDegD\subset  \FP^3\times  \FP^N$ with the following properties
\begin{itemize}
 \item $\ProjChowVarietyDegD$ and $\tildeProjChowVarietyDegD$ are (projective) constructible sets of complexity $O_D(1)$. To clarify, $\ProjChowVarietyDegD$ is defined by a Boolean combination of equalities and not-equalities of homogeneous polynomials in $K[x_0,\ldots,x_N]$, while $\tildeProjChowVarietyDegD$ is defined by a Boolean combination of equalities and not-equalities of polynomials that are homogeneous in $K[x_0,\ldots,x_N]$ and in $K[y_0,\ldots,y_3]$.
 \item For each irreducible curve $\gamma\subset \FP^3$, there is a unique point $z_{\gamma}\in \ProjChowVarietyDegD$ so that $\{z_{\gamma}\}\times\gamma\subset \tildeProjChowVarietyDegD$.
 \item Conversely, if $z\in\ProjChowVarietyDegD$, then $\tildeProjChowVarietyDegD\cap(\{z\}\times \FP^3)=\{z\}\times\gamma$ for some irreducible curve $\gamma$. Furthermore, $z=z_{\gamma}$.
\end{itemize}
\end{prop}
See e.g.~\cite{GM} for further details. For a friendlier introduction, one can also see \cite[Chapter 21]{harris} (in the notation used in \cite{harris},  $\ProjChowVarietyDegD$ is called the ``open'' Chow variety). \cite{harris} works over $\CC$ rather than over arbitrary fields, but the arguments are the same (at least provided $\operatorname{char} K > D$, which will always be the case for us).

For our purposes, it will be easier to work with affine varieties, so we will identify $K^N$ with the set $\FP^N\backslash H$, where $H$ is a generic hyperplane. In Theorem \ref{mainStructureThm}, we are given a finite collection $\finiteSetCurves$ of irreducible curves of degree $\leq D$. Thus the choice of hyperplane $H$ does not matter, provided that none of the curves correspond to points in the Chow variety that lie in $H$.

Finally, we will fix a coordinate chart and only consider those curves in the Chow variety that do not lie in the plane $\{[x_0:x_1:x_2:x_3]\colon\ x_0= 0\}$ (this corresponds to the coordinate chart $(x_1,x_2,x_3)\mapsto[1:x_1:x_2:x_3]$). Since no curves from $\finiteSetCurves$ correspond to curves that lie in $\{[x_0:x_1:x_2:x_3]\colon\ x_0= 0\}$, this restriction will not pose any difficulties. Let $\chowVarietyDegExactlyD$ be the ``modified'' Chow variety, which consists of all projective curves $\gamma\subset\FP^3$ that do not lie in the plane $\{[x_0:x_1:x_2:x_3]\colon\ x_0= 0\}$, and for which the Chow point in $\FP^N$ corresponding to $\gamma$ does not lie in the hyperplane $H$.

$\chowVarietyDegExactlyD$ is a constructible set that parameterizes (almost all) irreducible degree $D$ algebraic curves in $K^3$. Of course our definition of $\chowVarietyDegExactlyD$ depends on the choice of hyperplane $H$, but we will suppress this dependence, since the choice of $H$ will not matter for our results.

\begin{lemma}[Properties of the (affine) Chow variety of degree $D$ curves]\label{propertiesOfChowExactlyD}
Let $D\geq 1$ and fix a hyperplane $H\subset\FP^N$, where $N=N(D)$ (as in Proposition \ref{propertiesOfProjChow}). Then there are sets $\chowVarietyDegExactlyD \subset K^N$, $\tildeChowVarietyDegExactlyD \subset K^N\times K^3$ with the following properties
\begin{itemize}
 \item $\chowVarietyDegExactlyD$ and $\tildeChowVarietyDegExactlyD$ are constructible sets of complexity $O_D(1)$.
 \item For each irreducible degree $D$ curve $\gamma\subset K^3$ whose projectivization does not correspond to a chow point in $H$, there is a unique point $z_{\gamma}\in\chowVarietyDegExactlyD$ so that $\{z_{\gamma}\}\times\gamma\subset \tildeChowVarietyDegExactlyD$.
 \item Conversely, if $z\in\chowVarietyDegExactlyD$, then $\tildeChowVarietyDegExactlyD\cap(\{z\}\times K^3)=\{z\}\times\gamma$ for some irreducible curve $\gamma$. Furthermore, $z=z_{\gamma}$.
\end{itemize}
\end{lemma}

The constructible set $\chowVarietyDegExactlyD$ parameterizes the set of irreducible curves of degree (exactly) $D$. However, Theorem \ref{degDCurvesThm} is a statement about curves of degree at most $D$. To deal with this, we will define a constructible set that parameterizes the set of curves of degree at most $D$. Recall that for $j=1,\ldots,D$, we have constructible sets $\tilde{\mathcal{C}}_{3,j} \subset K^{N_j}$ and $\tilde{\mathcal{C}}_{3,j}^* \subset K^{N_j}\times K^3$. Let $V = K^{N_1}\times\ldots\times K^{N_D}$ and let $W = (K^{N_1}\times K^3)\times\ldots\times (K^{N_D}\times K^3)$. Each copy of $K^{N_j}$ has a natural embedding into $V$, given by
\begin{equation*}
 \begin{split}
  &\varphi_j\colon K^{N_j}\to V,\\
  &x\mapsto (0,\ldots,0,x,0,\ldots,0)\in K^{N_1}\times\ldots\times K^{N_{j-1}}\times K^{N_j}\times K^{N_{j+1}}\times\ldots\times K^{N_D}.
 \end{split}
\end{equation*}

Similarly, there is a natural embedding $\varphi_j^*$ of $K^{N_j}\times K^3$ into $W$. Define 
\begin{equation*}
\begin{split}
\chowVarietyDegD &\defeq \bigcup_{j=1}^D \varphi_j(\tilde{\mathcal{C}}_{3,j}),\\
\tildeChowVarietyDegD &\defeq \pi\Big(\bigcup_{j=1}^D \varphi_j^*(\tilde{\mathcal{C}}_{3,j}^*)\Big),
\end{split}
\end{equation*}
where $\pi\colon W\to V\times K^{3}$ is the projection that identifies each copy of $K^3$ in the vector space $W=(K^{N_1}\times K^3)\times\ldots\times(K^{N_D}\times K^3)$ (note that there are many ways that these copies of $K^3$ can be identified, since there is no distinguished coordinate system for $K^3$. However, it doesn't matter which identification we choose).

Abusing notation slightly, we will call $\chowVarietyDegD$ the Chow variety of curves of degree (at most) $D$. From this point onwards, we will never refer to either the (projective) Chow variety or the Chow variety of curves of degree exactly $D$. The sets $\chowVarietyDegD$ and $\tildeChowVarietyDegD$ satisfy properties analogous to those of $\chowVarietyDegExactlyD$ and $\tildeChowVarietyDegExactlyD$. We will record these properties here.
\begin{lemma}[Properties of the (affine) Chow variety]\label{propertiesOfChowExactlyD}
Let $D\geq 1$ and fix hyperplanes $H_1\subset\FP^{N_1},\ldots,H_D\subset \FP^{N_D}$. Then there are sets $\chowVarietyDegD \subset K^N$, $\tildeChowVarietyDegD \subset K^N\times K^3$ with the following properties
\begin{itemize}
 \item $\chowVarietyDegD$ and $\tildeChowVarietyDegD$ are constructible sets of complexity $O_D(1)$.
 \item For each irreducible curve $\gamma\subset K^3$ of degree $j\leq D$ whose projectivization does not correspond to a chow point in $H_j$, there is a unique point $z_{\gamma}\in\chowVarietyDegD$ so that $\{z_{\gamma}\}\times\gamma\subset \tildeChowVarietyDegD$.
 \item Conversely, if $z\in\chowVarietyDegD$, then $\tildeChowVarietyDegD\cap(\{z\}\times K^3)=\{z\}\times\gamma$ for some irreducible curve $\gamma$ of degree at most $D$. Furthermore, $z=z_{\gamma}$.
\end{itemize}
\end{lemma}
\subsection{Surfaces doubly ruled by curves}
In this section we will give some brief definitions that allow us to say what it means for a surface to be doubly ruled by curves.
\begin{defn}\label{defnOfDoublyRuled}
Let $K$ be an algebraically closed field, let $Z\subset K^3$, let $D\geq 1,$ and let $\mathcal{C}\subset\chowVarietyDegD$ be a constructible set. We say that $Z$ is doubly ruled by curves from $\mathcal{C}$ if there is a Zariski open set $O\subset Z$ so that for every $x\in O,$ there are at least two curves from $\mathcal{C}$ passing through $x$ and contained in $Z$.
\end{defn}
Surfaces that are doubly ruled by curves have many favorable properties. The proposition below details some of them.
\begin{prop}\label{implicationsOfDoublyRuled}
Let $K$ be an algebraically closed field, let $Z\subset K^3$ be an irreducible surface, let $D\geq 1$, and let $\mathcal{C}\subset\chowVarietyDegD$ be a constructible set. Suppose that $Z$ is doubly ruled by curves from $\mathcal{C}$. Then
\begin{itemize}
 \item $\deg(Z)\leq100 D^2.$
 \item For any $t\geq 1$, we can find two families of curves from $\mathcal{C}$, each of size $t$, so that each curve from the first family intersects each curve from the second family.
\end{itemize}
\end{prop}
\begin{rem}
In our definition of doubly ruled, we did not require that the curves passing through $x\in Z$ vary regularly as the basepoint $x\in Z$ changed. However, we get a version of this statement automatically. More precisely, the set
 \begin{equation}\label{curvesInZEqn}
  \{(x,\gamma)\in Z\times\mathcal{C}\colon x\in\gamma,\ \gamma\subset Z\}
 \end{equation}
 is a constructible set. Furthermore, there is a Zariski-open set $O\subset Z$ so that for all $x\in O$, the fiber of $\pi\colon \eqref{curvesInZEqn}\to Z$ contains at least two points. For example, if $K=\CC$, this means that we can find a Zariski-open set $O^\prime\subset O\subset Z$ so that as the basepoint $x\in O^\prime$ changes, we can smoothly (in the Euclidean topology) select two distinct curves $\gamma_{1,x},\ \gamma_{2,x}$ passing through $x$ that are contained in $Z$.
\end{rem}

\begin{defn}\label{AcontainedC}
Let $k$ be a field, and let $K$ be the algebraic closure of $k$. Let $D\geq 1$ and let $\mathcal{C}\subset\chowVarietyDegD$ be a constructible set. Let $\finiteSetCurves$ be a finite set of irreducible curves of degree at most $D$ in $k^3$. Abusing notation, we say that $\finiteSetCurves\subset\mathcal{C}$ if $\hat\gamma$ is an element of $\mathcal{C}$ for each $\gamma\in\finiteSetCurves$. Here $\hat\gamma$ is the Zariski closure (in $K$) of $\iota(\gamma)$, where $\iota\colon k\to K$ is the obvious embedding. For example, if $\gamma\subset\RR^3$ is a real curve, then $\hat\gamma$ is the complexification of $\gamma$, i.e.~the smallest complex curve whose real locus is $\gamma$.
\end{defn}
\subsection{Statement of the theorem}
We are now ready to state a precise version of the theorem alluded to in the paragraph following Theorem \ref{degDCurvesThm}.
\begin{thm}\label{mainStructureThm}
Fix $D>0, C>0$. Then there are constants $c_1,C_1,C_2$ so that the following holds. Let $k$ be a field and let $K$ be the algebraic closure of $K$. Let $\mathcal{C}\subset\chowVarietyDegD$ be a constructible set of complexity at most $C$. Let $\finiteSetCurves$ be a collection of $n$ irreducible algebraic curves in $k^3$, with $\finiteSetCurves\subset\mathcal{C}$ (see Definition \ref{AcontainedC}). Suppose furthermore that $\operatorname{char}(k)=0$ or $n\leq c_1 (\operatorname{char}(k))^2$.

Then for each number $A>C_1n^{1/2}$, at least one of the following two things must occur
\begin{itemize}
 \item There are at most $C_2An$ points in $k^3$ that are incident to two or more curves from $\finiteSetCurves$.
 \item There is an irreducible surface $Z\subset k^3$ that contains at least $A$ curves from $\finiteSetCurves$. Furthermore, $\hat Z$ is doubly ruled by curves from $\mathcal{C}$. See Definition \ref{defnOfDoublyRuled} for the definition of doubly ruled, and see Proposition \ref{implicationsOfDoublyRuled} for the implications of this statement.
\end{itemize}
\end{thm}
Taking $\mathcal{C}=\chowVarietyDegD,$ we see that Theorem \ref{degDCurvesThm} is an immediate corollary of Theorem \ref{mainStructureThm}. It remains to prove Theorem \ref{mainStructureThm}.
\section{Curves and complete intersections}\label{PQSec}

An algebraic curve $\gamma \subset K^3$ is a complete intersection if $\gamma = Z(P) \cap Z(Q)$ for some $P, Q \in K[x_1, x_2, x_3]$.  Not every algebraic curve is a complete intersection, but any algebraic curve $\gamma$ is contained in a complete intersection. Complete intersections are easier to work with in some situations, and we will often study a curve $\gamma$ using a complete intersection $Z(P) \cap Z(Q) \supset \gamma$.  In this subsection, we discuss the space of complete intersections, and we show that any algebraic curve lies in a complete intersection with some convenient properties.

We let $K[x_1, x_2, x_3]_{\le D} \subset K[x_1, x_2, x_3]$ be the space of polynomials of degree at most $D$.  We will sometimes abbreviate this space as $K[x]_{\le D}$.  The space $K[x]_{\le D}$ is a vector space of dimension ${\binom{D+3}{3}}$.  We choose an identification of $K[x]_{\le D}$ with $K^{\binom{D+3}{3}}$.

We use the variable $\alpha$ to denote an element of $K[x]_{\le D}^2$, and we write

$$\alpha = (P_\alpha, Q_\alpha) \in K[x]_{\le D}^2 = \left( K^{\binom{D+3}{3}} \right)^2.$$

Given an irreducible curve $\gamma$, we look for a choice of $\alpha \in K[x]_{\le D}^2$ so that $\gamma \subset Z(P_\alpha) \cap Z(Q_\alpha)$ and where $P_\alpha$ and $Q_\alpha$ have some other nice properties.

One useful property has to do with regular points.  Recall that a point $x \in \gamma$ is called regular (or smooth) if there are two polynomials $f_1, f_2 \in I(\gamma)$ so that $\nabla f_1(x)$ and $\nabla f_2(x)$ are linearly independent (cf. \cite{harris} Chapter 14, page 174.)  If $x$ is a regular point, then we will want to choose $\alpha$ so that $\nabla P_\alpha(x)$ and $\nabla Q_\alpha(x)$ are linearly independent.  We formalize these properties in a definition.

\begin{defn}\label{defnOfAssociated}
Let $\gamma\in \chowVarietyDegD$. We say that a point $\alpha\in \big(K^{\binom{D+3}{3}}\big)^2$ is \emph{associated} to $\gamma$ if $\gamma\subset\BZ(P_\alpha)\cap\BZ(Q_\alpha)$. If $x\in\gamma$ is a regular point of $\gamma$, we say that $\alpha$ is \emph{associated} to $\gamma$ at $x$ if $\alpha$ is associated to $\gamma$ and $\nabla P_\alpha(x)$ and $\nabla Q_\alpha(x)$ are linearly independent.
\end{defn}

Finally, given a surface $Z = Z(T) \subset K^3$, with $\gamma$ not contained in $Z$, we would like to choose $P_\alpha$ and $Q_\alpha$ so that $Z \cap Z(P_\alpha) \cap Z(Q_\alpha)$ is 0-dimensional.  The following Lemma says that we can choose $\alpha \in K[x]_{\le D}$ with all these good properties:

\begin{lemma}[Trapping a curve in a complete intersection]\label{trappingCurveLem}
Let $Z=Z(T)\subset K^3$ be a surface. Let $\gamma\in \chowVarietyDegD$. If $\gamma$ is not contained in $Z$, then there exists $\alpha\in\big(K^{\binom{D+3}{3}}\big)^2$ associated to $\gamma$ so that $Z\cap\BZ(P_\alpha)\cap\BZ(Q_\alpha)$ is 0 dimensional. We say that $\alpha$ is \emph{associated} to $\gamma$ and \emph{adapted} to $Z$.

Moreover, if $x \in \gamma$ is a regular point, we can also arrange that $\alpha$ is associated to $\gamma$ at $x$.
\end{lemma}

\begin{proof}
Suppose $w\in K^3\backslash 0$. Let $w^\perp$ be the two--plane passing through 0 orthogonal to $w,$ i.e.~$w^\perp$ has the defining equation $\{x\cdot w=0\}.$  Let $\pi_w: K^3 \rightarrow w^\perp$ be the orthogonal projection ($\pi_w (x) = x - (x \cdot w) w$).
For $x\in K^3$, let $L_{x,w}=\{x+aw\colon a\in K\}$ be the line in $K^3$ passing through the point $x$ and pointing in the direction $w$.  The fibers of the map $\pi_w$ are lines of the form $L_{x,w}$.

If $\zeta \subset K^3$ is a curve, then $\pi_w (\zeta) \subset w^\perp$ is a constructible set of dimension at most 1.  For a generic $w$, $\zeta$ does not contain any line of the form $L_{x,w}$, and in this case, $\pi_w(\zeta)$ is infinite and so $\pi_w(\zeta)$ is a constructible set of dimension 1: a curve with a finite set of points removed.

We can find the polynomials $P_1, P_2$ in the following way.  We pick two vectors $w_1, w_2 \in K^3$, and we consider $\pi_{w_i} (\gamma)$.  We let
the Zariski-closure of $\pi_{w_i} (\gamma)$ be $Z(p_i)$, where $p_i$ is a polynomial on $w^\perp$.  Then we let $P_i = p_i \circ \pi_{w_i}$ be the corresponding polynomial on $K^3$.  It follows immediately that $\gamma \subset Z(P_1) \cap Z(P_2)$.  For a generic choice of $w_1, w_2$, we will see that the pair of polynomials $(P_1, P_2)$ has all the desired properties.

First we discuss the degree of $P_1$ and $P_2$.  For generic vectors $w_i$, the degree of each polynomial $p_i$ is equal to the degree of $\gamma$.  This happens because the degree of $p_i$ is equal to the number of intersection points between $\pi_{w_i}(\gamma)$ with a generic line in $w_i^\perp$ and the degree of $\gamma$ is equal to the number of intersection points between $\gamma$ and a generic plane in $K^3$.  (cf. Chapter 18 of \cite{harris}, pages 224-225.)  But for a line $\ell \subset w_i^\perp$, the number of intersection points between $\pi_{w_i}(\gamma)$ and the line $\ell$ is equal to the number of intersection points between $\gamma$ and the plane $\pi_{w_i}^{-1}(\ell)$.

For any given surface $Z = Z(T)$, we check that for a generic choice of $w_i$, $Z(P_i) \cap Z$ is 1-dimensional.  Suppose that $Z' \subset Z \cap Z(P_i)$ is a 2-dimensional surface: so $Z'$ is an irreducible component of $Z$ and of $Z(P_i)$.  If $x \in Z'$, then the line $L_{x, w_i}$ must lie in $Z'$ also.  In particular, $w_i$ must lie in the tangent space $T_x Z' $.  But each component $Z'$ of $Z$ must contain a smooth point $x$, and a generic vector $w_i$ does not lie in $T_x Z'$.  So for a generic pair of vectors $w_1, w_2$, the pairwise intersections $Z \cap Z(P_1)$, $Z \cap Z(P_2)$, and $Z(P_1) \cap Z(P_2)$ are all 1-dimensional.

Next we consider the triple intersection $Z \cap Z(P_1) \cap Z(P_2)$.  Let $\zeta_1$ be the curve $Z \cap Z(P_1)$.   Suppose $\gamma$ is not contained in $Z$.  The curve $\zeta_1$ has irreducible components $\zeta_{1,1}, \zeta_{1,2}, ..., \zeta_{1,\ell}$, none of which is $\gamma$.  For a generic choice of $w_2$, the curves $\pi_{w_2} (\zeta_{1,j})$ and $\pi_{w_2}(\gamma)$ intersect properly (in a 0-dimensional subset of $w_2^\perp$).  Therefore, $Z(p_2)$, the Zariski closure of $\pi_{w_2}(\gamma)$, does not contain any of the images $\pi_{w_2}(\zeta_{1,j})$.  Hence $Z(P_2)$ does not contain any of the curves $\zeta_{1,j}$, and so $Z \cap Z(P_1) \cap Z(P_2)$ is 0-dimensional.

Now suppose that $x \in \gamma$ is a smooth point of $\gamma$.  For a generic choice of $w_i$, $\pi_{w_i}(x)$ will be a smooth point of $\pi_{w_i}(\gamma)$.  In this situation, $\nabla p_i( \pi_{w_i}(x) ) \not= 0$, and so $\nabla P_i(x) \not= 0$.  Let $v$ be a non-zero vector in $T_x (\gamma)$.  The vector $\nabla P_i(x)$ must be perpendicular to $v$ (because $\gamma \subset Z(P_i)$), and it must be perpendicular to $w_i$ (because the line $L_{x, w_i} \subset Z(P_i)$).  Therefore, $\nabla P_i(x)$ is a (non-zero) multiple of the cross-product $w_i \times v$.   If we also assume that $w_1, w_2$ and $v$ are linearly independent, then it follows that $\nabla P_1(x)$ and $\nabla P_2(x)$ are linearly independent, and so $\nabla P_1(x) \times \nabla P_2(x) \not= 0$.
\end{proof}

The choice of $\alpha$ in the Lemma above is not unique, and we also want to keep track of the set of $\alpha$ with various good properties.

Let
\begin{equation}
G \defeq \big\{(x,\alpha)\in K^3\times\big(K^{\binom{D+3}{3}}\big)^2\colon P_{\alpha}(x)=Q_{\alpha}(x)=0,\ \nabla P_\alpha(x) \times\nabla Q_\alpha(x)\neq 0\big\}.
\end{equation}
This is a constructible set of complexity $O_D(1)$.
\begin{lemma}\label{xAlphaGammaConst}
The sets
\begin{align}
&\big\{(\gamma,\alpha)\in \chowVarietyDegD\times\big(K^{\binom{D+3}{3}}\big)^2\colon \gamma\subset\BZ(P_\alpha)\cap\BZ(Q_\alpha)\},\label{AlphaGammaSet}\\
&\big\{(x,\gamma,\alpha)\in K^3\times \chowVarietyDegD\times \big(K^{\binom{D+3}{3}}\big)^2\colon (x,\alpha)\in G,\ \gamma\subset \BZ(P_\alpha)\cap \BZ(Q_\alpha)\big\}\label{xAlphaGammaSet}
\end{align}
are constructible and have complexity $O_{D}(1)$.
\end{lemma}
\begin{proof}  The proof uses the fact that constructible sets are a Boolean algebra as well as Chevalley's theorem, Theorem \ref{ChevalleyTheorem}.

The following sets are constructible of complexity $O_D(1)$:
\begin{align}
&\{(x,\gamma)\in K^3\times\chowVarietyDegD\colon x\in\gamma\}\quad\quad \textrm{(by Proposition \ref{propertiesOfProjChow})},\nonumber\\
&\{(x,\alpha)\in K^3\times \big(K^{\binom{D+3}{3}}\big)^2\colon P_\alpha(x)=Q_\alpha(x)=0\},\nonumber\\
&\{(x,\alpha)\in K^3\times \big(K^{\binom{D+3}{3}}\big)^2\colon (P_\alpha(x)\neq 0\ \textrm{or}\ Q_\alpha(x)\neq 0\},\nonumber\\
&\{(x,\gamma,\alpha)\in K^3\times \chowVarietyDegD\times \big(K^{\binom{D+3}{3}}\big)^2\colon x\in\gamma,\  P_\alpha(x)\neq 0\ \textrm{or}\ Q_\alpha(x)\neq0\}.\label{xAlphaGammaNotContained}
\end{align}
Let $\pi\colon (x,\gamma,\alpha)\mapsto(\gamma,\alpha)$. Then
\begin{equation}\label{gammaNotInAlpha}
\pi(\eqref{xAlphaGammaNotContained})=\{(\gamma,\alpha)\in \chowVarietyDegD\times\big(K^{\binom{D+3}{3}}\big)^2\colon \gamma\not\subset \BZ(P_\alpha)\cap\BZ(Q_\alpha) \},
\end{equation}
so
\begin{equation*}
\eqref{AlphaGammaSet}=\Big(\chowVarietyDegD\times\big(K^{\binom{D+3}{3}}\big)^2\Big)\backslash\eqref{gammaNotInAlpha}
\end{equation*}
is constructible of complexity $O_{D}(1)$.

Finally,
\begin{equation*}
\eqref{xAlphaGammaSet}= \Big(K^3\times \eqref{AlphaGammaSet}\Big)\cap\{(x,\gamma,\alpha)\in K^3\times \chowVarietyDegD\times\big(K^{\binom{D+3}{3}}\big)^2\colon (x,\alpha)\in G\}
\end{equation*}
is constructible of complexity $O_{D}(1)$.
\end{proof}
\section{Local Rings, Intersection multiplicity and B\'ezout}
\subsection{Local rings}
\begin{defn} For $z\in K^N$, let $\mathcal{O}_{K^N,z}$ be the local ring of $K^N$ at $z$. This is the ring of rational functions of the form $p(x)/q(x)$, where $q(z)\neq 0$.
\end{defn}
There is a natural map $\iota\colon K[x_1,\ldots,x_N]\to \mathcal{O}_{K^N,z}$ which sends $f\mapsto f/1$. This map is an injection.

\begin{defn}\label{defnOfLocalizationOfIdeal}
If $I\subset K[x_1,\ldots,x_N]$ is an ideal, let $I_{z}\subset \mathcal{O}_{K^N,z}$ be the localization of $I$ at $z$; this is the ideal in $\mathcal{O}_{K^N,z}$ generated by $\iota(I)$.
\end{defn}

\begin{lemma}\label{trivialLocalRing}
If $Z\subset K^N$ is an affine variety and if $z\notin Z$, then $I(Z)_{z}=\mathcal{O}_{K^N,z}$.
\end{lemma}
\begin{proof}
Since $z\notin Z$, there is a function $f\in I(Z)$ which is non-zero at $z$. Then the element $f/1\in I(Z)_{z}$ is a unit.
\end{proof}
\begin{lemma}\label{equalityOfLocalizedIdeals}
Let $Z\subset K^N$ be an affine variety. Suppose that $Y$ is an irreducible component of $Z$ and $z\in Y$ is a regular point of $Z$. Then $I(Y)_{z}=I(Z)_{z}$.
\end{lemma}
\begin{proof}
Write $Z = Y\cup Z_1\cup\ldots\cup Z_\ell$ as a union of irreducible components. We get a corresponding decomposition $I(Z)_{z}=I(Y)_{z}\cap I(Z_1)_{z}\cap\ldots\cap I(Z_\ell)_{z}$. Since $z$ is a regular point of $Z$, for each index $j$ we have $z\notin Z_j$ and thus by Lemma \ref{trivialLocalRing}, $I(Z_j)_{z}=\mathcal{O}_{K^N,z}$. Thus $I(Z)_{z}=I(Y)_{z}$.
\end{proof}
\subsection{The B\'ezout theorem}  \label{secbezout}

In the paper, we will need a few variations of the B\'ezout theorem.  One of the versions involves the multiplicity of the intersection of hypersurfaces $Z(f_1) \cap ... \cap Z(f_N)$.  We start by defining this multiplicity.

Given polynomials $f_1,\ldots,f_N$, we define the (length) intersection multiplicity of $f_1,\ldots,f_N$ at $z$ to be
\begin{equation}\label{defnMultz0}
\mult_{z}(f_1,\ldots,f_N) = \dim_K\big(\mathcal{O}_{K^N,z}/(f_1,\ldots,f_N)_{z}\big).
\end{equation}
Let us understand \eqref{defnMultz0}.
\begin{itemize}
\item $\mathcal{O}_{K^N,z}$ is the local ring of $K^N$ at $z$.
\item $(f_1,\ldots,f_N)_{z}$ is the ideal in $\mathcal{O}_{K^N,z}$ generated by $(f_1,\ldots,f_N)$. I.e.~it is the set of elements of $\mathcal{O}_{K^N,z}$ which can be written $a_1f_1+\ldots+a_Nf_N$, with $a_1,\ldots,a_N\in \mathcal{O}_{K^N,z}$.
\item $\mathcal{O}_{K^N,z}/(f_1,\ldots,f_N)_z$ is the set of equivalence classes of elements in $\mathcal{O}_{K^N,z}$ where $\bar g \sim \bar g^\prime$ if $g-g^\prime\in (f_1,\ldots,f_N)_z$. This set of equivalence classes forms a ring.
\item $\dim_K(\cdot)$ is the dimension of the ring $\mathcal{O}_{K^N,z}/(f_1,\ldots,f_N)_z$ when it is considered as a $K$--vector space. Later, $\dim(\cdot)$ (without the subscript) will be used to denote the dimension of an algebraic variety or constructible set.
\end{itemize}

If $z \in \BZ(f_1)\cap\ldots\cap \BZ(f_N)$, then the multiplicity $\mult_{z}(f_1,\ldots,f_N)$ is always at least 1.  Later, we will show that the multiplicity of certain intersections is large. To do this, we need to find many linearly independent elements of $\mathcal{O}_{K^N,z}/(f_1,\ldots,f_N)_z$. Polynomials $g_1(x),\ldots,g_\ell(x)$ are linearly dependent in $\mathcal{O}_{K^N,z}/(f_1,\ldots,f_N)_z$ if we can write
\begin{equation}\label{gEllLinDependenceRelation}
\sum_{i=1}^{\ell} c_i g_i(x) = a_1(x)f_1(x)+\ldots+a_N(x)f_N(x),
\end{equation}
where $c_1,\ldots,c_{\ell}\in K;$ at least one $c_i$ is non-zero; and $a_1(x),\ldots,a_N(x)\in\mathcal{O}_{K^N,z}$, i.e. $a_1(x),\ldots,a_N(x)$ are rational functions of the form $p(x)/q(x)$ with $q(z)\neq 0$.

If no expression of the form \eqref{gEllLinDependenceRelation} holds for $g_1,\ldots,g_\ell$, then $g_1,\ldots,g_{\ell}$ are \emph{linearly independent}.

We can now state the first version of B\'ezout's theorem that we will use.

\begin{thm}[B\'ezout's theorem for surfaces in $K^3$]\label{bezoutHypersurfaces}
Let $Z_1=Z(f_1),Z_2=Z(f_2),Z_3=Z(f_3)$ be surfaces in $K^3$. Suppose that $Z_1\cap Z_2\cap Z_3$ is zero-dimensional (i.e.~finite). Then
\begin{equation}
 \sum_{z\in Z_1\cap Z_2\cap Z_3}\mult_{z}(f_1,f_2,f_3)\leq (\deg f_1)(\deg f_2)(\deg f_3).
\end{equation}
\end{thm}
This is a special case of \cite[Proposition 8.4]{Fulton}. Specifically, see Equation (3) on p145.

We will also need a version of B\'ezout's theorem that bounds the number of (distinct) intersection points between a curve and a surface in $K^3$.

\begin{thm}[B\'ezout's theorem for curves and surfaces in $K^3$; see \cite{Fulton}, Theorem 12.3]\label{bezoutCurveSurface}
Let $Z\subset K^3$ be a surface and let $\gamma\subset K^3$ be an irreducible curve. Then either $\gamma\subset Z$ or $\gamma$ intersects $Z$ in at most $(\deg\gamma)(\deg Z)$ distinct points.
\end{thm}

Finally, we will need a version of B\'ezout's theorem that bounds the number of curves in the intersection between two surfaces in $K^3$.

\begin{thm}\label{bezoutcurves} Suppose that $f_1, f_2 \in K[x_1, x_2, x_3]$.  If $f_1$ and $f_2$ have no common factor, then the number of irreducible curves in $Z(f_1) \cap Z(f_2)$ is at most $(\deg f_1) (\deg f_2)$.
\end{thm}

\begin{proof} We will prove this result using Theorem \ref{bezoutHypersurfaces}.  Suppose that $\{ \gamma_i \}_{i = 1,  ..., M}$ is a finite set of irreducible curves in $Z(f_1) \cap Z(f_2)$.  We want to show that $M \le (\deg f_1) (\deg f_2)$.  If $\pi$ is a generic plane in $K^3$, then $\pi$ will intersect each of the curves $\gamma_i$ (cf. \cite{harris} Corollary 3.15).  There are only finitely many points that lie in at least two of the curves $\gamma_i$, and a generic plane $\pi$ will avoid all of those points.  Therefore, $| \pi \cap Z(f_1) \cap Z(f_2) | \ge M$.  Let $f_3$ be a polynomial of degree 1 with $Z(f_3) = \pi$.

Since $f_1$ and $f_2$  have no common factor, $Z(f_1) \cap Z(f_2)$ must have dimension 1.  Then for a generic plane $\pi$, $Z(f_1) \cap Z(f_2) \cap Z(f_3)$ has dimension zero, so we can apply the B\'ezout theorem, Theorem \ref{bezoutHypersurfaces}.  In this way we see that,

$$ M \le | Z(f_1) \cap Z(f_2) \cap Z(f_3) | \le \sum_{z\in Z_1\cap Z_2\cap Z_3}\mult_{z}(f_1,f_2,f_3) \le $$

$$ \leq (\deg f_1)(\deg f_2)(\deg f_3) = (\deg f_1)(\deg f_2). $$

\end{proof}

\section{Curve-surface tangency}

In this section, we will define what it means for a curve to be tangent to a surface to order $r$. A precise definition will be given in Section \ref{defnCurveSurfaceTangency}.

As a warmup, suppose that the curve $\gamma$ is the $x_1$-axis.  In this case, we could make the definition that $\gamma$ is tangent to the surface $Z(T)$ at the origin to order at least $r$ if and only if

$$ \frac{ \partial^j T}{ \partial x_1^j} (0) = 0 \textrm{ for } j = 0, ..., r. $$

\noindent The definition we give will be equivalent to this one in the special case that $\gamma$ is the $x_1$-axis.  One of our goals is to extend this definition to any regular point $x$ in any curve $\gamma$.  To do so, we define a version of  ``differentiating along the curve $\gamma$'' in Subsection \ref{subsecdiffcurve}.  In the following subsection, we give two other definitions of being tangent to order $r$, and we show that all the definitions are equivalent.

Throughout the section, we will restrict to the case that $r < \chara K$.  To see why, suppose again that $\gamma$ is the $x_1$-axis, and consider the polynomial $T = x_2 - x_1^p$ for $p = \chara K$.  For this choice of $T$, $ \frac{ \partial^j T}{ \partial x_1^j} (0) = 0$ for all $j$.  Nevertheless, it does not seem correct to say that the $x_1$-axis is tangent to $Z(T)$ to infinite order, and with our other definitions of tangency, the $x_1$-axis is not tangent to $Z(T)$ to infinite order.  To avoid these issues, we restrict throughout this paper to the case $r < \chara K$.

\subsection{Differentiating along a curve} \label{subsecdiffcurve}
Recall the definition of $P_\alpha(x)$ and $Q_{\alpha}(x)$ from Section \ref{PQSec}. Here and throughout this section, $\nabla = \nabla_x=(\partial_{x_1},\partial_{x_2},\partial_{x_3})$ denotes the gradient in the $x$ variable.

We define a differential operator $D_\alpha$.  For any $f$, we define

\begin{equation}
D_\alpha f (x)\defeq\big(\nabla P_\alpha(x)\times\nabla Q_\alpha(x)\big) \cdot \nabla f(x).
\end{equation}

This is well-defined for $f \in K[x_1, x_2, x_3]$, and in this case, $D_\alpha f \in K[x_1, x_2, x_3]$.  The operator $D_\alpha$ also makes sense a little more generally: if $f$ is a rational function, then $D_\alpha f$ is a rational function as well.

The intuition behind $D_\alpha$ is the following.  If $\gamma \subset Z(P_\alpha) \cap Z(Q_\alpha)$, and $x \in \gamma$ is a point where $\nabla P_\alpha(x) \times \nabla Q_\alpha(x) \not= 0$, then $\nabla P_\alpha(x) \times \nabla Q_\alpha(x)$ is tangent to $\gamma$, and so $D_\alpha f(x)$ is a derivative of $f$  in the tangent direction to $\gamma$.

We let $D_\alpha^2 f$ be shorthand for $D_\alpha ( D_\alpha f)$, and we define the higher iterates $D_\alpha^j f$ in a similar way.

Note that $D_\alpha P_{\alpha}$ and  $D_\alpha Q_{\alpha}$ are identically 0 (as functions of $\alpha$ and $x$). Thus $(D_\alpha^j P_\alpha)(x)=0$ and $(D_\alpha^j Q_\alpha)(x)=0$ for all $j\geq 1$. If $P_\alpha(x)=0$ and $Q_{\alpha}(x)=0$, then $(D_\alpha^j P_\alpha)(x)=0$ and $(D_\alpha^j Q_\alpha)(x)=0$ for all $j\geq 0.$

We also observe that $D_\alpha(x)$ obeys the Leibniz rule. In particular, if $f(x),\ g(x)$ are polynomials or rational functions, then
\begin{equation}\label{productRule}
D_\alpha(fg)(x)=(D_\alpha f)(x)\ g(x)+f(x)\ (D_\alpha g)(x).
\end{equation}

As a corollary, we have the following result.
\begin{lemma}\label{derivativesProducts}
Let $T\in K[x_1,x_2,x_3]$ and suppose $(D_\alpha^jT)(x)=0,\ j=0,\ldots,r$. Then for any rational function $b$ with $b(x)\neq\infty$, we have $D_\alpha^j (bT)(x)=0,\ j=0,\ldots,r.$
\end{lemma}
\subsection{Defining curve-surface tangency}\label{defnCurveSurfaceTangency}
\begin{defn} \label{defI_zr}
For $z\in K^3$ and $r\geq 0$, let $I_{z,\geq r}$ be the ideal of polynomials in $K[x_1,x_2,x_3]$ that vanish at $z$ to order $\geq r$.
\end{defn}

For example, if $z = 0$, then $I_{z, \geq r}$ is the ideal generated by the monomials of degree $r$.

\begin{defn}\label{defnOfTangency} Let $T\in K[x_1,x_2,x_3]$. Let $\gamma\subset K^3$ be an irreducible curve and let $z\in \gamma$ be a regular point of $\gamma$. Let $r<\operatorname{char}(K)$.  We say that \emph{$\gamma$ is tangent to $Z(T)$ at $z$ to order $\geq r$} if
$T/1 \in (I(\gamma))_{z}+(I_{z,\geq r+1})_z$.
\end{defn}
\begin{rem}
Definition \ref{defnOfTangency} abuses notation slightly, since $\gamma$ being tangent to $Z(T)$ depends on the polynomial $T$, not merely its zero-set $Z(T)$. This would be a natural place to use the language of schemes, but we will refrain from doing so to avoid introducing more notation.
\end{rem}
For example, the $x$-axis is tangent to the surface $y = x^{r+1}$ at the origin to order $r$.

We now show that this definition is equivalent to several other definitions.  In particular, we will see that being tangent to order $r$ can also be defined using the tangential derivatives $D_\alpha$.

\begin{thm}\label{definitionsOfTangencyAreEquivalentThm}
Let $T\in K[x_1,x_2,x_3]$. Let $\gamma\subset K^3$ be an irreducible curve and let $z\in \gamma$ be a regular point of $\gamma$. Let $r<\operatorname{char}(K)$.  Suppose that $\alpha$ is associated to $\gamma$ at $z$ as in Definition \ref{defnOfAssociated}.  Then the following are equivalent:

\begin{enumerate}[label=(\roman{*}), ref=(\roman{*})]
\item\label{TInLocalIdeal} $\gamma$ is tangent to $Z(T)$ at $z$ to order $\ge r$.  I.e. $T/1 \in (I(\gamma))_{z}+(I_{z,\geq r+1})_z$.
\item\label{Alpha} $D_\alpha^j T(z)=0,\ j=0,\ldots,r.$
\item\label{TInIdeal} $T\in I(\gamma)+I_{z,\geq r+1}$.
\end{enumerate}
\end{thm}

Before we prove the theorem, we note the following consequence.  Condition \ref{Alpha} depends on the choice of $\alpha$, but the other conditions don't.  Therefore, we get the following corollary:

\begin{cor} \label{oneaalla}  Let $T\in K[x_1,x_2,x_3]$. Let $\gamma\subset K^3$ be an irreducible curve and let $z\in \gamma$ be a regular point of $\gamma$. Let $r<\operatorname{char}(K)$.

Suppose that there exists one $\alpha$ associated to $\gamma$ at $z$ so that $D_\alpha^j T(z)=0,\ j=0,\ldots,r.$  Then for every $\alpha$ associated to $\gamma$ at $z$, $D_\alpha^j T(z)=0,\ j=0,\ldots,r.$
\end{cor}

\begin{proof}[Proof of Theorem \ref{definitionsOfTangencyAreEquivalentThm}]  We will prove that $\ref{TInIdeal} \Longrightarrow \ref{TInLocalIdeal} \Longrightarrow \ref{Alpha} \Longrightarrow \ref{TInIdeal}$.
It is straightforward to see that $\ref{TInIdeal}\Longrightarrow\ref{TInLocalIdeal}$: just localize both ideals at $z$.\\

$\ref{TInLocalIdeal} \Longrightarrow \ref{Alpha}$: First, note that if $\alpha$ is associated to $\gamma$ at $z$, then by Lemma \ref{equalityOfLocalizedIdeals}, $I(\gamma)_z = (P_\alpha,Q_\alpha)_z$. In particular, if $T\in (I(\gamma))_z+(I_{z,\geq r+1})_z$ then $T=AP_\alpha+BQ_\alpha +C$, where $A,B\in \mathcal{O}_{K^3,z}$ and $C\in (I_{z,\geq r+1})_z$. By Lemma \ref{derivativesProducts} (and the observation that $(D_\alpha^j P_\alpha)(z)=0$ and $(D_\alpha^j Q_\alpha)(z)=0$ for all $j\geq 0$), we conclude that $D_\alpha^j(T)=0,\ j=0,\ldots,r$.\\

$\ref{Alpha}\Longrightarrow\ref{TInIdeal}$: Consider the map
\begin{equation*}
\begin{split}
E\colon &K[x_1,x_2,x_3] \to K^{r+1},\\
&T \mapsto \big(T(z),D_\alpha T(z),\ldots, D_\alpha^{r}T(z)\big).
\end{split}
\end{equation*}
(Here $z$ was specified in the statement of Theorem \ref{definitionsOfTangencyAreEquivalentThm}). $E$ is a linear map. By Corollary \ref{derivativesProducts}, $(P_\alpha,Q_\alpha)$ is in the kernel of $E$. $I_{z,\geq r+1}$ is also in the kernel of $E$. Let $V \defeq K[x_1,x_2,x_3]/\big((P_\alpha,Q_\alpha)+I_{z,\geq r+1}\bigr)$. Then $\tilde E\colon V \to K^{r+1}$ is well-defined.

We will show that $\tilde E$ is an isomorphism.  By $\ref{Alpha}$, $T$ is in the kernel of $E$.  Since $\tilde E$ is an isomorphism, we see that

$$T \in (P_\alpha,Q_\alpha)+I_{z,\geq r+1} \subset I(\gamma) + I_{z \geq r+1}.$$

\noindent This will show that $\ref{Alpha} \Longrightarrow \ref{TInIdeal}$.  It only remains to check that $\tilde E$ is an isomorphism.  To do this, we will show that $E$ is surjective and we will show that $\dim_{K}(V) \le \dim_{K}(K^{r+1})=r+1$.

Since $\alpha$ is associated to $\gamma$ at $z$, $\nabla P_\alpha(z)\times\nabla Q_\alpha(z)\neq0$. Without loss of generality we can assume that $(1,0,0)\cdot \nabla P_\alpha(z)\times\nabla Q_\alpha(z)\neq0$ (indeed, if this fails then we can replace $(1,0,0)$ with $(0,1,0)$ or $(0,0,1)$, and permute indices accordingly).

\begin{lemma}\label{propertiesOfR}
Let $\sigma(x)=\pi_1(x-z),$ where $\pi_1\colon K^3\to K$ is the projection to the first coordinate. Then for any $j < \chara K$, $(D_\alpha^i \sigma^j)(z)=0,\ i=0,\ldots,j-1$, and $(D_\alpha^j\sigma)(z)\neq 0$.
\end{lemma}
\begin{proof}
We can expand $D_\alpha^i \sigma^j$ using the Leibniz rule. If $i<j$, then every term in the expansion will contain a factor of the form $(\pi_1(x-z))^{j-i}$, which evaluates to 0 when $x=z$.

Conversely, we have
\begin{equation*}
D_\alpha^j \sigma^j(x) = \big(\pi_1(\nabla P_\alpha(x)\times \nabla Q_\alpha(x))\big)^jj! + \textrm{terms containing a factor of the form}\ \pi_1(x-z).
\end{equation*}
Evaluating at $x=z$, we conclude
\begin{equation*}
\begin{split}
D_\alpha^j \sigma^j(z) &= \big(\pi_1(\nabla P_\alpha(z)\times \nabla Q_\alpha(z))\big)^jj! \\
&\neq 0.
\end{split}
\end{equation*}
Here we used the assumption that $j <\operatorname{char}(K)$, so $j!\neq 0$, and the assumption that $\pi_1(\nabla P_\alpha(z)\times \nabla Q_\alpha(z))\neq 0$, so $\big(\pi_1(\nabla P_\alpha(z)\times \nabla Q_\alpha(z))\big)^j\neq 0.$
\end{proof}
Lemma \ref{propertiesOfR} implies that the $(r+1) \times (r+1)$ matrix
\begin{equation*}
 \left[\begin{array}{c} E(1) \\ E(\sigma)\\ E(\sigma^2) \\ \vdots \\ E(\sigma^{r}) \end{array}\right]
\end{equation*}
is upper-triangular and has non-zero entries on the diagonal. In particular, $E$ is surjective, and so $\tilde E$ is surjective.

It remains to check that $\dim_{K}(V) \le r+1$.  Since $\nabla P_\alpha(z)\times\nabla Q_{\alpha}(z)\neq 0$, $\nabla P_\alpha(z)$ and $\nabla Q_{\alpha}(z)$ are linearly independent (and in particular, both are non-zero). Thus after a linear change of coordinates, we can assume that $z$ is the origin, $P(x_1,x_2,x_3) = x_2 + P^*(x_1,x_2,x_3)$, and $Q(x_1,x_2,x_3) = x_3 + Q^*(x_1,x_2,x_3)$, with $P^*,Q^*\in I_{z,\geq 2}$.

We will study the successive quotients $I_{z, \ge s} / I_{z, \ge s+1}$.  We note that $I_{z, \ge s} / I_{z, \ge s+1}$ is isomorphic (as a vector space) to the homogeneous polynomials of degree $s$.

\begin{lemma} \label{PQvs23} For any $s$,

$$ \frac{(P, Q) \cap I_{z, \ge s}}{ I_{z, \ge s+1}} \supset \frac{(x_2, x_3) \cap I_{z, \ge s}}{ I_{z, \ge s+1}}.$$

\end{lemma}

\begin{proof}.  Suppose that $R \in (x_2, x_3) \cap I_{z, \ge s}$.  Since $(x_2, x_3)$ is a homogeneous ideal, the degree $s$ part of $R$ must lie in $(x_2, x_3)$, and so we get

$$ R = R_2 x_2 + R_3 x_3 + R',$$

\noindent where $R_2, R_3$ are homogeneous polynomials of degree $s-1$, and $R' \in I_{z, \ge s+1}$.  Therefore, $R = R_2 P + R_3 Q + R''$, where $R'' \in I_{z, \ge s+1}$.  \end{proof}

Therefore, we see that

$$ \dim \frac{ (P, Q) \cap I_{z, \ge s} }{I_{z, \ge s+1}}  \ge \dim \frac{ (x_2, x_3) \cap I_{z, \ge s}}{I_{z, \ge s+1}}= \dim \frac{I_{z, \ge s}}{ I_{z, \ge s+1}} - 1. $$

These dimensions are sufficient to reconstruct $\dim V $.
We write $(P,Q)_{\ge s}$ for $I_{z, \ge s} \cap (P,Q)$.  We first note that

\begin{equation} \label{dimV} \dim V = \dim \frac{ I_{z, \ge 0}}{ (P,Q) + I_{z, \ge r+1} } = \sum_{s=0}^r \dim \frac{ I_{z, \ge s} }{(P,Q)_{\ge s} + I_{z, \ge s+1}}. \end{equation}

Next, we note the short exact sequence

$$ \frac{ (P, Q)_{\ge s}}{ I_{z, \ge s+1} }\rightarrow \frac{ I_{z, \ge s}}{I_{z, \ge s+1}} \rightarrow \frac {I_{z, \ge s}}{ (P, Q)_{\ge s} + I_{z, \ge s+1}}. $$

Using Lemma \ref{PQvs23} and this short exact sequence, we see that

$$ \dim \frac{ I_{z, \ge s} }{(P, Q)_{\ge s} + I_{z, \ge s+1} }\le 1. $$

Plugging this estimate into equation \ref{dimV}, we get

$$ \dim V =  \dim \frac{ I_{z, \ge 0}}{(P,Q) + I_{z, \ge r+1} } \le \sum_{s=0}^r 1 = r+1. $$

\end{proof}

\section{Curve-surface tangency and intersection multiplicity}
In this section we will show that if a curve $\gamma$ is tangent to $Z(T)$ at $z$ to order $\geq r$, then the varieties $Z(T)$ and $\gamma$ intersect at $z$ with high multiplicity.  We defined the intersection multiplicity for a complete intersection in Section \ref{secbezout}.  For a curve $\gamma$, we consider a complete intersection $\BZ(P_{\alpha})\cap\BZ(Q_\alpha) \supset \gamma$, where $\alpha$ is associated to $\gamma$ at $z$ and adapted to $Z$ (see Definition \ref{defnOfAssociated} and Lemma \ref{trappingCurveLem}).

\begin{lemma}\label{DControlsIntMult}
Let $T\in K[x_1,x_2,x_3]$, $\gamma\in\chowVarietyDegD$, and $z\in K^3.$ Suppose that $z$ is a regular point of $\gamma$ and $\gamma\not\subset Z(T)$. Let $r\leq\operatorname{char}(K)$ and suppose that $\gamma$ is tangent to $Z(T)$ at $z$ to order $\geq r$.

Then for any $\alpha$ that is associated to $\gamma$ at $z$ and adapted to $Z(T)$, we have
\begin{equation}
\operatorname{mult}_{z}(P_\alpha,Q_\alpha,T)\geq r+1.
\end{equation}
\end{lemma}
Before we prove Lemma \ref{DControlsIntMult}, we will state a key corollary
\begin{cor}[Very very tangent implies trapped]\label{verVerTangentImpliesTrapped}
Let $T\in K[x_1,x_2,x_3]$, $\gamma\in\chowVarietyDegD$, and $z\in K^3.$ Suppose that $D^2(\deg T)<\operatorname{char}(K)$ and $z$ is a regular point of $\gamma$. Suppose that $\gamma$ is tangent to $Z(T)$ at $z$ to order $\geq D^2(\deg T)$. Then $\gamma\subset Z(T)$.
\end{cor}
\begin{proof}[Proof of Corollary \ref{verVerTangentImpliesTrapped}]
Suppose $\gamma\not\subset Z(T)$. Let $\alpha$ be associated to $\gamma$ at $z$ and adapted to $Z(T)$, as in Definition \ref{defnOfAssociated} and Lemma \ref{trappingCurveLem}.  In particular, $Z(T)\cap\BZ(P_\alpha)\cap\BZ(Q_\alpha)$ is a zero-dimensional set.  Lemma \ref{trappingCurveLem} also guarantees that $P_\alpha$ and $Q_\alpha$ have degree at most $D$.  By Lemma \ref{DControlsIntMult}, $\operatorname{mult}_{z}(P_\alpha,Q_\alpha,T)> D^2(\deg T)$. But this contradicts B\'ezout's theorem (Theorem \ref{bezoutHypersurfaces}). Thus we must have $\gamma\subset Z(T)$.
\end{proof}

\begin{proof}[Proof of Lemma \ref{DControlsIntMult}]
By Theorem \ref{definitionsOfTangencyAreEquivalentThm}, we have $D_\alpha^j T(z)=0,\ j=0,\ldots,r$. We also have $D_\alpha^jP_\alpha(z)=0$ and $D_\alpha^jQ_\alpha(z)=0$ for all $j\geq 0$.

\begin{lemma}\label{powersOfRAreLinIndependent}
Let $\sigma(x)=\pi_1(x-z)$, as in Lemma \ref{propertiesOfR}. $1,\sigma(x),\ldots,\sigma^{r}(x)$ are linearly independent elements of $\mathcal{O}_{K^3,z}/(P_\alpha,Q_\alpha,T)_z$.
\end{lemma}
\begin{proof}
Suppose this was not the case. Then recalling \eqref{gEllLinDependenceRelation}, there must exist a linear dependence relation of the form
\begin{equation}\label{notLinIndep}
\sum_{i=0}^{r}d_i\sigma(x)^i = a(x)P_\alpha(x) + b(x)Q_\alpha(x)+ c(x)T(x),
\end{equation}
where $a(x),b(x),c(x)$ are rational functions of $x$ that are not $\infty$ when $x=z$,  and $\{d_i\}$ are elements of $K$, not all of which are 0.

Let $j$ be the smallest index so that $d_j\neq 0$. Then, using Lemma \ref{propertiesOfR}, we get
\begin{equation*}
\begin{split}
D_\alpha^{j}\Big( \sum_{i=0}^{r-1}d_i\sigma(z)^i\Big) &= D_\alpha^j(d_j\sigma(z)^j) +\sum_{j<i\leq r-1} D_\alpha^j(d_i\sigma(z)^i)\\
&=D_\alpha^j (d_j\sigma(z)^j)\\
&\neq 0.
\end{split}
\end{equation*}
On the other hand,
\begin{equation*}
D_\alpha^{j}(aP_\alpha)(z)+D_\alpha^{j}(bQ_\alpha)(z)+D_\alpha^{j}(cT)(z)=0,
\end{equation*}
which is again a contradiction. Thus \eqref{notLinIndep} cannot hold.
\end{proof}

From Lemma \ref{powersOfRAreLinIndependent}, we conclude that
\begin{equation}
\operatorname{dim}\big(\mathcal{O}_{K^3,z}/(P_\alpha,Q_\alpha,T)_z\big)\geq r+1.
\end{equation}
Lemma \ref{DControlsIntMult} now follows from the definition of multiplicity from \eqref{defnMultz0}.
\end{proof}
\section{Sufficiently tangent implies trapped}\label{suffTangTrappedSec}
\begin{thm}\label{tangentImpliesTrappedProp}
Let $D\geq 0.$ Then there exists $c_1>0$ (small) and $r_0$ (large) so that the following holds. Let $K$ be a closed field and let $T \in K[x_1, x_2, x_3]$ be an irreducible polynomial with $\deg T < c_1\operatorname{char}(K)$. Then there is a (non-empty) open subset $O\subset Z(T)$ with the following property: if $\gamma \subset K^3$ is an irreducible curve of degree at most $D$, $z\in O$ is a regular point of $\gamma$, and $\gamma$ is tangent to $Z(T)$ at $z$ to order $\geq r_0$, then $\gamma\subset Z(T)$.
\end{thm}
\begin{rem}
For example, suppose $K=\CC$, $D=1$ (so $\gamma$ is a line), and $T = x_1 - x_2^u$, where $u$ is a very large integer (much larger than $r_0$). Then $O\subset Z(T)$ is the compliment of the set $\{x_1=0,x_2=0\}.$ We will not calculate the value of $r_0$ corresponding to $D=1$; however, any number $r_0\geq 2$ would suffice in this case.

At any point $z\in Z(T) \backslash O$, if $\gamma$ is a line tangent to $Z(T)$ at $z$ to order $\geq 2$, then $\gamma\subset Z(T)$.  The only such lines are lines of the form $x_1 = a_1; x_2 = a_2$ for some $(a_1, a_2)$ satisfying $a_1 - a_2^u = 0$.
On the other hand, any line passing through a point in $\{x_1=0,x_2=0\}$ and tangent to the 2--plane $x_1\cdot z = 0$, is tangent to $Z$ to order $u \gg r_0$. Most of these lines will not be contained in $Z$.  This does not contradict the proposition, since $\{x_1=0,x_2=0\}$ lies outside the set $O$.
\end{rem}

\subsection{Defining the tangency functions}
Let $T\in K[x_1,x_2,x_3]$ be an irreducible polynomial.  Recall that $\alpha\in\big(K^{\binom{D+3}{3}}\big)^2$ parameterizes pairs of polynomials $(P_\alpha, Q_\alpha)$ of degree at most $D$, as described in Section \ref{PQSec}.

 For each $j\geq 0,$ define
\begin{equation}\label{defnOfHj}
h_{j}(\alpha, x) \defeq (D_\alpha^j T)(x) \in K[\alpha, x]
\end{equation}

\begin{lemma}[Properties of the functions $h_{j,\alpha}$]\label{existsHiThm}
The polynomials $h_{j}$ defined above have the following properties.
\begin{enumerate}[label=(\roman{*}), ref=(\roman{*})]
\item $h_{j}(\alpha, x)$ is a polynomial in $\alpha$ and $x$. Its degree in $\alpha$ is $O_j(1)$, and its degree in $x$ is at  most $\deg T < c_1 \operatorname{char}(K)$.
\item $h_{0}(\alpha, x)=T(x)$.
\item\label{hmeastang} Let $\gamma$ be an irreducible curve of degree at most $D$.  If $z$ is a regular point of $\gamma$, $\alpha$ is associated to $\gamma$ at $z$, and $r < \operatorname{char}(K)$, then $\gamma$ is tangent to $Z(T)$ at $z$ to order $r$ if and only if
$$ h_j(\alpha, z) = 0 \qquad \textrm{for}\ j=0,\ldots,r. $$
\item\label{existsHiThmProp3} Let $\gamma$ be an irreducible curve of degree at most $D$. If $z$ is a regular point of $\gamma$, $\alpha$ is associated to $\gamma$ at $z$, and
\begin{equation}\label{hjAlphaIs0}
h_{j}(\alpha, z)=0\qquad \textrm{for}\ j=0,\ldots,D^2(\deg T),
\end{equation}
then $\gamma\subset Z$.
\end{enumerate}
\end{lemma}
\begin{proof}
The first two properties follow immediately from the definition of $h_{j}$. The third property is Theorem \ref{definitionsOfTangencyAreEquivalentThm}.  For the last property, \eqref{hjAlphaIs0} implies that $\gamma$ is tangent to $Z$ at $z$ to order $\geq D^2(\deg T)$.  By choosing $c_1(D)$ small enough, we can assume that $D^2 \deg T < \operatorname{char}(K)$.  We now apply Corollary \ref{verVerTangentImpliesTrapped}.
\end{proof}
We would like to find an open set $O \subset Z(T)$ so that for $z \in O$, if $h_j(\alpha, z) = 0$ for $j$ up to some $r_0$, then $h_j(\alpha, z)$ is forced to vanish for many more $j$.   This forcing comes from a quantitative version of the ascending chain condition.

\subsection{A quantitative Ascending Chain condition}

To set up the right framework to apply the ascending chain condition, we introduce the field of fractions of $Z(T)$.

\begin{defn}
Let
\begin{equation*}
\FZT=\Big\{p/q\colon p,q\in K[x_1,x_2,x_3] / (T),\ q\not= 0 \Big\}
\end{equation*}
be the field of rational functions on $\BZ(T)$.

Let $\rho_T$ be the map $K[x]\to \FZT.$  We also write $\rho_T$ for the corresponding map $K[\alpha, x] \to \FZT[\alpha]$.
\end{defn}

\begin{defn}
Let $\tilde h_j =\rho_T(h_j) \in \FZT[\alpha]$.
\end{defn}

We note that $\tilde h_j$ is a polynomial of degree $O_{j}(1)$ in the variable $\alpha$.

\begin{defn}
Let $\tilde K$ be a field, and let $I\subset \tilde K[y_1,\ldots,y_N]$ be an ideal. We define
\begin{equation*}
\complexity(I)=\min(\deg f_1+\ldots+\deg f_\ell),
\end{equation*}
where the minimum is taken over all representations $I=(f_1,\ldots,f_\ell)$.
\end{defn}
%\TODO{Replace the statement with a more general one involving ideals}
\begin{prop}\label{quantACCProp}
Let $\tilde K$ be a field, let $N\geq 0,$ and let $\tau\colon \NN\to\NN$ be a function. Then there exists a number $M_0$ with the following property. Let $\{I_i\}$ be a sequence of ideals in $\tilde K[y_1,\ldots,y_N]$, with $\complexity(I_i)\leq \tau(i).$ Then there exists a number $r_0\leq M_0$ so that $I_{r_0}\in(I_1,\ldots,I_{r_0-1})$.
\end{prop}
To avoid interrupting the flow of the argument, we will defer the proof of this Proposition to Appendix \ref{AQCCApp}. We will use the following special case of Proposition \ref{quantACCProp}, which we will state separately:
\begin{cor}\label{quantACCCor}
Let $\tilde K$ be a field, let $N\geq 0,$ and let $\tau\colon \NN\to\NN$ be a function. Then there exists a number $M_0$ with the following property. Let $\{f_i\}$ be a sequence of polynomials in $\tilde K[y_1,\ldots,y_N]$, with $\deg(f_i)\leq \tau(i).$ Then there exists a number $r_0\leq M_0$ so that $f_{r_0}\in(f_1,\ldots,f_{r_0-1})$.
\end{cor}

%The purpose of this section is to prove the following lemma.
\begin{lemma}\label{gNoInTermsOfTileGi} There exists a number $r_0$ which is bounded by some constant $C(D)$ and a sequence $\tilde a_i \in \FZT[\alpha]$ so that we have the equality
\begin{equation}\label{expressTildeGn0}
\tilde h_{r_0}=\sum_{i=0}^{r_0-1} \tilde a_i \tilde h_i.
\end{equation}
This equation holds in $\FZT[\alpha]$.
\end{lemma}

\begin{proof}
 This is Corollary \ref{quantACCCor} with $f_i = \tilde h_i(\alpha)$ and $\tau(i)=\deg\tilde h_i = O_{i}(1)$, $\tilde K = \FZT$, $N = 2 { D+3 \choose 3}$, and $y = \alpha$.
\end{proof}

Our next goal is to rewrite equation \ref{expressTildeGn0} in terms of $h_{j}$ instead of $\tilde h_j$.  To do this, we recall the localization of $K[x]$ at $T$.

\begin{defn}
Let
\begin{equation*}
K[x]_T \defeq \Big\{p/q\colon p,q\in K[x_1,x_2,x_3],\ q\notin (T) \Big\}.
\end{equation*}

$K[x]_T$ is called the localization of $K[x]$ at $T$.  It is a ring.

\end{defn}

The map $\rho_T: K[x]\to \FZT$  extends in a natural way to a ring homomorphism $K[x]_T \to \FZT$.  It is surjective: given any $\tilde p \in K[x_1, x_2, x_3] / (T)$, and $0 \not= \tilde q \in K[x_1, x_2, x_3] / (T)$, we can pick representatives $p \in K[x_1, x_2, x_3]$ and $q \in K[x_1, x_2, x_3] \setminus (T)$, and then $p/q \in K[x]_T$ and $\rho_T (p/q) = \tilde p / \tilde q \in \FZT$.

We write $K[x]_T [\alpha]$ for the ring of polynomials in $\alpha$ with coefficients in the ring $K[x]_T$.  (Writing $K[x]_T [\alpha]$ looks a little funny because of all the brackets, but this is the standard notation: if $R$ is a ring, then $R[\alpha]$ are the polynomials in $\alpha$ with coefficients in $R$, and our ring $R$ is $K[x]_T$.)  We also write $\rho_T$ for the natural extension $K[x]_T [\alpha] \rightarrow \FZT[\alpha]$, which is also surjective.

We note that $\rho_T(h_i) = \tilde h_i$.  We pick $a_i \in K[x]_T [\alpha]$ so that $\rho_T (a_i) = \tilde a_i$.  We can now rewrite Equation \ref{expressTildeGn0} in terms of $h_i, a_i$:

$$ \rho_T \left( h_{r_0} - \sum_{i=0}^{r_0-1} a_i h_i \right) = 0. $$

The kernel of $\rho_T: K[x]_T [\alpha] \rightarrow \FZT[\alpha]$ is the set

$$ \{ b T : b \in K[x]_T [\alpha] \}. $$

Therefore, we can choose $b \in K[x]_T [\alpha]$ so that

$$ h_{r_0}= b T + \sum_{i=0}^{r_0-1} a_i h_i. $$

Finally, we recall that $T = h_0$. Therefore, after changing the definition of $a_0$, we can arrange that

\begin{equation}\label{expressGn0}
 h_{r_0} =\sum_{i=0}^{r_0-1}a_i h_{i}.
\end{equation}

In this equation, $r_0 \le C(D)$, $a_i \in K[x]_T [\alpha]$, and the equation holds in $K[x]_T [\alpha]$.

\subsection{Trapping the curves}

\begin{lemma}\label{allGjDefined}
Let the functions $\{h_i\}$ be as defined in \eqref{defnOfHj}, and suppose that \eqref{expressGn0} holds. Then for all $j=0,1,\ldots,$ we can choose $a_{i,j} \in K[x]_T [\alpha]$ so that
\begin{equation}\label{expressionForHi}
h_{j} = \sum_{i=0}^{r_0-1}a_{i,j} h_{i}.
\end{equation}
\end{lemma}

\begin{proof}
We will prove Lemma \ref{allGjDefined} by induction on $j$. The case $j\leq r_0-1$ is immediate. The case $j=r_0$ is precisely \eqref{expressGn0}. Now assume the theorem has been proved up to some value $j$.  To do the induction step, we will apply the operator $D_\alpha$ to the equation for $j$ in order to get the equation for $j+1$.

We note that the ring $K[x]_T$ is closed under the action of the partial derivatives $\partial_i$.  Therefore, $K[x]_T [\alpha]$ is closed under the action of $D_\alpha$.
We will make liberal use of the Leibniz rule \eqref{productRule} for the operator $D_\alpha$, which we recall here: for any $f,g \in K[x]_T [\alpha]$,

\begin{equation*}
D_\alpha(fg)=(D_\alpha f)g+ f(D_\alpha g).
\end{equation*}

We also recall that $D_\alpha h_j = h_{j+1}$.   Therefore, we have
\begin{equation*}
 \begin{split}
  h_{j+1} &= D_\alpha h_{j} = D_\alpha \left( \sum_{i=0}^{r_0 - 1} a_{i,j} h_i \right) \\
  &=\sum_{i=0}^{r_0-1} \left( (D_\alpha a_{i,j}) h_i + a_{i,j} (D_\alpha h_i) \right) \\
  &=\sum_{i=0}^{r_0-1} (D_\alpha a_{ij}) h_{i} + \sum_{i=0}^{r_0-2}a_{i,j} h_{i+1}+a_{r_0-1,j}h_{r_0}  \\
  &=\sum_{i=0}^{r_0-1} (D_\alpha a_{i,j}) h_i + \sum_{i=1}^{r_0-1}a_{i-1,j} h_{i} + a_{r_0-1,j}\sum_{i=0}^{r_0-1} a_i h_i \\
&=\sum_{i=0}^{r_0-1}a_{i,j+1} h_i,
 \end{split}
\end{equation*}
where
\begin{equation}
a_{i,j+1} = D_\alpha a_{i,j}+a_{i-1,j} + a_{r_0-1,j} a_i ,
\end{equation}
and $a_{-1,j}(x)=0$ for each index $j$. This completes the induction.
\end{proof}
We can now define the open set $O \subset Z(T)$.  Let $M=D^2\deg(T)$.  The set $O$ is the subset of $Z(T)$ where none of the denominators involved in $a_{i,j}$ vanishes for $j \le M$.  Let us spell out what this means more carefully.  Each $a_{i,j} \in K[x]_T [\alpha]$.  Therefore, we can write as a finite combination of monomials in $\alpha$:

$$ a_{i,j} = \sum_I r_{i,j,I} \alpha^I. $$

\noindent In this sum, $I$ denotes a multi-index, and $r_{i,j,I} \in K[x]_T$.  For each $i,j$, there are only finitely many values of $I$ in the sum.  Each $r_{i,j,I}$ is a rational function $p_{i,j,I} / q_{i,j,I}$, where $q_{i,j,I} \notin (T)$.  We let $O \subset Z(T)$ be the set where none of the denominators $q_{i,j,I}$ vanishes, for $0 \le i \le r_0-1$ and $1 \le j \le M$:

\begin{equation}
O=Z\ \backslash \bigcup_{\substack{j=1,\ldots,M\\i=0,\ldots,r_0-1}} \BZ(q_{i,j, I}).
\end{equation}

The set $O$ is non-empty by a standard application of the Hilbert Nullstellensatz.  Since $T$ is irreducible, the radical of $(T)$ is $(T)$.  Since $K$ is algebraically closed, we can apply the Nullstellensatz, and we see that the ideal of polynomials that vanishes on $Z(T)$ is exactly $(T)$.  We know that each denominator $q_{i,j,I} \notin (T)$.  Since $T$ is irreducible, $(T)$ is a prime ideal, and so $\prod q_{i,j,I} \notin (T)$.  Therefore, $\prod q_{i,j,I}$ does not vanish on $Z(T)$.  This shows that $O$ is not empty.

Now Lemma \ref{allGjDefined} has the following Corollary on the set $O$:

\begin{cor}\label{tangR0ImpliesTangM}
Suppose that $z\in O,\ \alpha \in\big(K^{\binom{D+3}{3}}\big)^2$, and
\begin{equation}
h_{j}(\alpha, z)=0,\quad j=0,\ldots, r_0-1.
\end{equation}
Then $h_{j}(\alpha, z)=0,\ j=0,\ldots,M$.
\end{cor}

\begin{proof} By Equation \ref{expressionForHi}, we know that for all $j \le M$,

$$h_{j} = \sum_{i=0}^{r_0-1}a_{i,j} h_{i}. $$

Expanding out the $a_{i,j}$ in terms of $p_{i,j,I}$ and $q_{i,j,I}$, we get

$$ h_j (\alpha, x) = \sum_{i=0}^{r_0-1} \left( \sum_I \frac{ p_{i,j,I}(x) }{ q_{i,j,I}(x) } \alpha^I \right)  h_{i} (\alpha, x). $$

At the point $x =z \in O$, the polynomials $q_{i,j,I}$ are all non-zero.  By assumption, $h_i (\alpha, z) = 0$ for $i=0, ..., r_0 - 1$.  Therefore, we see that $h_j (\alpha, z) = 0$ also.

\end{proof}

We can now prove Theorem \ref{tangentImpliesTrappedProp}. Let $\gamma$ be an irreducible curve of degree at most $D$. Let $z\in O$ be a smooth point of $\gamma$, and suppose $\gamma$ is tangent to $Z$ at $z$ to order $\geq r_0$.  Use Lemma \ref{trappingCurveLem} to choose an $\alpha$ associated to $\gamma$ at $z$, as in Definition \ref{defnOfAssociated}.  By Theorem \ref{definitionsOfTangencyAreEquivalentThm}, $h_{j,\alpha}(z)=0,\ j=0,\ldots,r_0$. But by Corollary \ref{tangR0ImpliesTangM}, this implies $h_{j,\alpha}(z)=0,\ j=1,\ldots, D^2(\deg T)$. Then by item $\ref{existsHiThmProp3}$ from Lemma \ref{existsHiThm}, $\gamma\subset Z$. This concludes the proof of Theorem \ref{tangentImpliesTrappedProp}.
\subsection{Trapped implies sufficiently tangent}
We will also need a converse to Theorem \ref{tangentImpliesTrappedProp}. We will call this property ``trapped implies sufficiently tangent.''
\begin{lemma}\label{DIsLargeWhenCommongComponent}
Let $T\in K[x_1,x_2,x_3]$ and let $\gamma\in \chowVarietyDegD$ with $\gamma\subset Z(T)$. Let $z$ be a regular point of $\gamma$ and let $\alpha$ be associated to $\gamma$ at $z$. Then

\begin{equation}
D^j_\alpha T(z)=0\quad\ \textrm{for all}\ j\geq 0.
\end{equation}
\end{lemma}

\begin{proof}
By Lemma \ref{equalityOfLocalizedIdeals}, $(T)_{z}=(P_\alpha,Q_\alpha)_{z}$. In particular, we can write $T=\frac{p_1}{q_1}P_\alpha+\frac{p_2}{q_2}Q_\alpha$, where $q_1(z)\neq 0,\ q_2(z)\neq 0$. By Lemma \ref{derivativesProducts}, we have that $D_\alpha^j T(z)=0$ for all $j$.
\end{proof}

\section{Generalized flecnodes and constructible conditions}\label{genFlecnodeSection}
Let $K$ be a closed field. Let $T\in K[x_1,x_2,x_3]$ with $\deg T < \operatorname{char}(K)$, and consider $\BZ(T) \subset K^3$.  We recall that a point $x \in Z(T)$ is called flecnodal if there is a line $L$ which is tangent to $Z(T)$ at $x$ to order at least 3.  We consider the following generalization:

Given a constructible set $\mathcal{C}\subset\chowVarietyDegD$, and given integers $t,r \ge 1$, with $r < \operatorname{char} K$, we say that a point $x \in K^3$ is $(t,\mathcal{C},r)$-flecnodal for $T$ if there are $\ge t$ distinct curves $\gamma_1,\ldots,\gamma_t\in  \mathcal{C}$ passing through the point $x$, so that $x$ is a regular point of each of these curves, and each of these curves is tangent (in the sense of Definition \ref{defnOfTangency}) to $\BZ(T)$ at $x$ to order $\geq r$.  The original definition of a flecnode corresponds to $t=1$; $\mathcal{C}=\mathcal{C}_{3,1}$, the space of all lines in $K^3$; and $r=3$.

The flecnode polynomial, discovered by Salmon, is an important tool for studying flecnodes.  For each $T$, Salmon constructed a polynomial $\Flec T$ of degree $\le 11 \Deg T$, so that a point $x \in \BZ(T)$ is flecnodal if and only if $\Flec T(x) = 0$.  Our goal is to generalize this result to $(t,\mathcal{C},r)$-flecnodal points.

Our theorem for $(t,\mathcal{C},r)$ flecnodes is a little more complicated to state, but it is almost equally useful in incidence geometry.  Instead of one polynomial $\Flec T$, we will have a sequence of polynomials $\Flec_j T$, where $j$ goes from 1 to a large constant $J(t,\mathcal{C},r)$.  To tell whether a point $x$ is flecnodal, we check whether $\Flec_j T(x)$ vanishes for $j= 1, ..., J$.  Based on that information, we can determine whether $x$ is $(t,\mathcal{C},r)$-flecnodal.  Here is the precise statement of the theorem.
\begin{thm} \label{genflec}
For each constructible set $\mathcal{C}\subset \chowVarietyDegD$ and each pair of integers $t,r\geq 1$ with $r < \operatorname{char} K$, there is an integer $J=J(t,\mathcal{C},r)$, and a subset $B_{F(t,\mathcal{C},r)} \subset \{0,1\}^J$ so that the following holds.  For each $1 \le j \le J$, and for each $T \in K[x_1, x_2, x_3]$, there are polynomials $\Flec_j T = \Flec_{t, \mathcal{C}, r, j} T \in K[x_1, x_2, x_3]$ so that
\begin{itemize}
\item $\Deg \Flec_{j} T \le C(t,\mathcal{C},r) \Deg T$.
\item $x$ is $(t,\mathcal{C},r)$-flecnodal for $T$ if and only if the vector $v(\Flec_{j} T(x) ) \in B_{F(t,\mathcal{C},r)} \subset \{ 0 , 1 \}^J$.
\end{itemize}
\end{thm}

(Recall from Section \ref{constSetSec} that $v (y) $ is zero if $y = 0$ and $1$ if $y \not= 0$.)

The main tool in the proof is Chevalley's quantifier elimination theorem, Theorem \ref{ChevalleyTheorem}.  The method is quite flexible and it can also be used to study other variations of the flecnode polynomial.

\subsection{The $r$-jet of a polynomial}

The first observation in the proof is that whether a point $z$ is $(t, \mathcal{C}, r)$-flecnodal for a polynomial $T$ only depends on the point $z$ and the $r$-jet of $T$ at $z$.  Recall that the $r$-jet of $T$ at $z$, written $J^r T_z$ is the polynomial of degree at most $r$ that approximates $T$ at $z$ to order $r$.  Here is the more formal definition.  Recall that for any point $z \in K^n$, $I_{z, \ge r} \subset K[x_1, ..., x_n]$ is the ideal of polynomials that vanish to order at least $r$ at the point $z$ -- see Definition \ref{defI_zr}.

\begin{defn} \label{rjet} For any $T \in K[x]$, the $r$-jet of $T$ at $z$, $J^r T_z$, is the unique polynomial of degree at most $r$ so that

$$ T - J^r T_z \in I_{z, \ge r+1}. $$

\end{defn}

Since we assumed $r < \chara K$, the $r$-jet $J^r T_z$ can be computed with a Taylor series in the usual way, summing over multi-indices $I$:

\begin{equation} \label{taylor} J^r T_z (x) = \sum_{ |I| \le r} \frac{1}{ I !} \nabla_I T(z) (x-z)^I. \end{equation}

\noindent For any multi-index $I$ with $|I| \le r$, $\nabla_I T(z) = \nabla_I J^r T_z (z)$.

We can now state our first observation as a formal lemma.

\begin{lemma} \label{rjetcheck} A point $z \in K^3$ is $(t, \mathcal{C}, r)$-flecnodal for a polynomial $T$ if and only if $z$ is $(t, \mathcal{C}, r)$-flecnodal for $J^r T_z$.
\end{lemma}

\begin{proof}  Suppose that $\gamma \in \mathcal{C}$ and $z$ is a regular point of $\gamma$.  It suffices to check that $\gamma$ is tangent to $Z(T)$ at $z$ to order $\ge r$ if and only if $\gamma$ is tangent to $Z( J^r T_z )$ to order $\ge r$.

By Theorem \ref{definitionsOfTangencyAreEquivalentThm}, $\gamma$ is tangent to $Z(T)$ at $z$ to order $\ge r$ if and only if

$$ T \in I(\gamma) + I_{z, \ge r+1}.$$

Similarly, $\gamma$ is tangent to $Z(J^r T_z)$ at $z$ to order $\ge r$ if and only if

$$ J^r T_z \in I(\gamma) + I_{z, \ge r+1}.$$

But $J^r T_z$ is defined so that $T - J_r T_z \in I_{z, \ge r+1}$, and so these conditions are equivalent.  \end{proof}

We now define the set $\Flec_{t, \mathcal{C}, r} \subset K^3 \times \Poly_{r} (K^3)$

$$\Flec_{t, \mathcal{C}, r} \defeq \{ (z, U) \in K^3 \times K[x]_{\le r} \textrm{ so that } z \textrm{ is } (t, \mathcal{C}, r) \textrm{-flecnodal for } U\}. $$

By Lemma \ref{rjetcheck}, $z$ is $(t, \mathcal{C}, r)$-flecnodal for $T$ if and only if $(z, J^r T_z) \in \Flec_{t, \mathcal{C}, r}$.

\subsection{Constructible conditions} \label{ConstructibleConditionsSection}

Given any subset $Y \subset K^3 \times K[x]_{\le r}$ we can think of $Y$ as a condition.  We say that $T$ obeys $Y$ at $z$ if and only if $(z, J^r T_z) \in Y \subset K^3 \times K[x]_{\le r}$.  If $Y$ is an algebraic set, we say that $Y$ is an algebraic condition, and if $Y$ is a constructible set, we say that $Y$ is a constructible condition.

We will prove below that $\Flec_{t, \mathcal{C}, r}$ is a constructible condition.  Any constructible condition $Y$ obeys a version of Theorem \ref{genflec}.  This follows immediately from the definition of a constructible set, as we now explain.

\begin{lemma} \label{construcnormalform}
Suppose that $Y \subset K^3 \times K[x]_{\le r}$ is a constructible condition.  Then for any polynomial $T: K^3 \rightarrow K$, there is a finite list of polynomials $Y_j T$, $j=1, ..., J(Y)$, and a subset $B_Y \subset \{0, 1\}^{J(Y)}$ obeying the following conditions:
\begin{itemize}
\item $\Deg Y_j T \le C(Y) \Deg T + C(Y)$.
\item The polynomial $T$ obeys condition $Y$ at a point $x$ if and only if $v(Y_j T(x)) \in B_Y$.
\end{itemize}
\end{lemma}
\begin{proof} Since $Y$ is a constructible set, there is a finite list of polynomials $f_j$ on $K^3 \times \Poly_k (K^3)$ so that $x \in Y$ if and only if $v(f_j(x) ) \in B_Y$.  By definition, $T$ obeys condition $Y$ at $x$ if and only if $ v(f_j( x, J^k T(x)) \in B_Y.$

We define $Y_j T(y) = f_j( y, J^k T(y))$.  So $T$ obeys condition $Y$ at $x$ if and only if $v(Y_j T(x)) \in B_Y$. Note that $J^k T$ is a vector-valued polynomial of degree $\le \Deg T$.  (Each coefficient of $J^k T$ is a constant factor times a derivative $\nabla_I T$ for some multi-index $I$, and each $\nabla_I T$ is a polynomial of degree $\le \Deg T$.)  We let $C(Y)$ be the maximal degree of the polynomials $f_j$.  Then $Y_j T$ is a polynomial of degree $\le C(Y) \Deg T + C(Y)$.
\end{proof}

Therefore, to prove Theorem \ref{genflec}, it only remains to show that $\Flec_{t, \mathcal{C}, r}$ is constructible.

\subsection{Checking constructibility}

We will now use Chevalley's theorem, Theorem \ref{ChevalleyTheorem}, to check that $\Flec_{t, \mathcal{C}, r}$ is constructible.  We build up to the set $\Flec_{t, \mathcal{C}, r}$ in a few steps, which we state as lemmas.

\begin{lemma}\label{Tx0GammaIsConstructible}
Let $\mathcal{C}\subset  \chowVarietyDegD$ be a constructible set of complexity $C$.
the set
\begin{equation}\label{TX0GammaAlphaSet}
\{(x,U,\gamma)\in K^3\times K[x]_{\le r}\times \mathcal{C} \colon x\in\gamma_{\smooth},\gamma\ \textrm{tangent to}\ Z(U)\ \textrm{at}\ x\ \textrm{to order}\ \geq r\}
\end{equation}
is constructible of complexity $O_{r,C,D}(1)$.
\end{lemma}
\begin{proof}
Recall that \eqref{xAlphaGammaSet} is the set of triples $(x, \gamma, \alpha)$ so that $\alpha$ is associated to $\gamma$ at $x$ (see Definition \ref{defnOfAssociated}).  By Lemma \ref{trappingCurveLem}, there exists an $\alpha$ so that $(x, \gamma, \alpha) \in \eqref{xAlphaGammaSet}$ if and only if $x \in \gamma_{\smooth}$.

By Theorem \ref{definitionsOfTangencyAreEquivalentThm}, $\gamma$ is tangent to $Z(U)$ at $x$ to order $\geq r$ if and only if $D_\alpha^jU(x)=0$ for each $j=0,\ldots,r$.

Consider the set
\begin{equation}\label{TGammaX0AlphaSet}
\begin{split}
\big\{(x,U,\gamma,\alpha)\in K^3\times K[x]_{\le r}\times \mathcal{C}&\times\big(K^{\binom{D+3}{3}}\big)^2\colon \\ &(x,\gamma,\alpha)\in \eqref{xAlphaGammaSet},D_\alpha^j U(x)=0,\ j=1,\ldots,r \big\}.
\end{split}
\end{equation}

Since \eqref{xAlphaGammaSet} is constructible, it is straightforward to check that this set is constructible.  Now let $\pi\colon (x,U,\gamma,\alpha)\mapsto(x,U,\gamma)$, and note that $\eqref{TX0GammaAlphaSet}=\pi(\eqref{TGammaX0AlphaSet})$.  By Theorem \ref{ChevalleyTheorem}, $\eqref{TX0GammaAlphaSet}=\pi(\eqref{TGammaX0AlphaSet})$ is constructible of complexity $O_{r,C,D}(1)$.
\end{proof}

\begin{cor}\label{manyTangentCurvesCor}
Let $\mathcal{C}\subset  \chowVarietyDegD$ be a constructible set of complexity $C$. The set
\begin{equation}\label{TxGamma1Gammat}
\begin{split}
\{(x,U,\gamma_1,\ldots,\gamma_t)\in &K^3\times K[x]_{\le r}\times \mathcal{C}^t \colon\\
&(x,U,\gamma_i)\in \eqref{TX0GammaAlphaSet}\ \textrm{for each}\ i=1,\ldots,t;\ \gamma_i\neq\gamma_j\ \textrm{if}\ i\neq j\}
\end{split}
\end{equation}
is constructible of complexity $O_{D,C,r,t}(1)$.
\end{cor}

\begin{proof} This follows from Lemma \ref{Tx0GammaIsConstructible}, using the fact that a Boolean combination of constructible sets is constructible.
\end{proof}

\begin{rem}
Though we will not need it here, one could also extend Corollary \ref{manyTangentCurvesCor} to a collection of constructible sets $\mathcal{C}_1,\ldots,\mathcal{C}_t\subset\chowVarietyDegD$. The version stated above is the special case $\mathcal{C}_1=\ldots=\mathcal{C}_t=\mathcal{C}$.
\end{rem}

\begin{cor} \label{flecconstr} If $1 \le t, r$ and $r < \chara K$, and if $\mathcal{C} \subset  \chowVarietyDegD$ is a constructible set of complexity $C$, then $\Flec_{t, \mathcal{C}, r}$ is constructible of complexity $O_{D,C,r,t}(1)$.
\end{cor}

\begin{proof}
Consider the projection
\begin{equation*}
\pi\colon (x,U,\gamma_1,\ldots,\gamma_t)\mapsto (x,U)\in K^3\times K[x]_{\le r}.
\end{equation*}

We note that $\Flec_{t, \mathcal{C}, r} = \pi ( \eqref{TxGamma1Gammat} )$.  By Corollary
\ref{manyTangentCurvesCor} and Chevalley's theorem, $\Flec_{t, \mathcal{C}, r}$ is constructible of complexity $O_{D,C,r,t}(1)$.

\end{proof}
\section{Being flecnodal is contagious}\label{flecnodeContagiousSection}
In this section, we will explore a corollary of Theorem \ref{genflec}. We will prove that an algebraic surface with ``too many'' $(t,\mathcal{C},r)$-flecnodal points must be $(t,\mathcal{C},r)$-flecnodal almost everywhere.
\begin{defn}\label{almostEveryPoint}
 We say that a condition holds at almost every point of a variety $Z$ if the subset of $Z$ where the condition fails is contained in a subvariety of lower dimension.  For example, a polynomial $T$ obeys $Y$ at almost every point of a curve $\gamma$ if and only if $Y$ holds at all but finitely many points of $\gamma$.
\end{defn}

\begin{prop} \label{contagiousflecnodes} For each $C, D, t,r$, there is a constant $C_1$ so that the following holds. Let $\mathcal{C}\subset\chowVarietyDegD$ be a constructible set of complexity $\leq C$. Suppose that $T \in K[x_1, x_2, x_3]$ is irreducible, and that $\BZ(T)$ contains at least $C_1 (\Deg T)^2$ algebraic curves in $\mathcal{C}$, each of which contains at least $C_1 \Deg T$ $(t,\mathcal{C},r)$-flecnodal points.  Then there is a Zariski open subset of $\BZ(T)$ consisting of $(t,\mathcal{C},r)$-flecnodal points.
\end{prop}

Proposition \ref{contagiousflecnodes} follows from the fact that \emph{all} constructible conditions are contagious---they all obey an estimate similar to that in Proposition \ref{contagiousflecnodes}.
\begin{lemma} \label{construccontag1} Suppose that $Y \subset K^3 \times K[x]_{\le r}$ is a constructible condition.  Then there is a constant $C(Y)$ so that the following holds.  Suppose that $\gamma \in \mathcal{C}\subset\chowVarietyDegD$.  Suppose that $T: K^3 \rightarrow K$ is a polynomial.  If $T$ obeys condition $Y$ at $>  C(Y) D (\Deg T + 1)$ points of $\gamma$, then $T$ obeys condition $Y$ at all but finitely many points of $\gamma$.
\end{lemma}
\begin{proof}
Let $z_1, z_2, ...$ be points of $\gamma$ where $T$ obeys $Y$.  We let $Y_j T$ be the polynomials described in Lemma \ref{construcnormalform}, for $j=1, ..., J(Y)$.  Recall that there is some set $B_Y \subset \{0,1\}^J(Y)$ so that $T$ obeys $Y$ at $z$ if and only if the vector $v( Y_j T(z) ) \in B_Y$.
In particular, at each point $z_k$, we have $v (Y_j T(z_k)) \in B_Y$.  There are $\le 2^{J(Y)}$ elements of $B_Y$.  By the pigeon-hole principle, we can choose an element $\beta \in B_Y \subset \{0, 1\}^{J(Y)}$ so that $v( Y_j T(z_k) ) = \beta$ holds for at least $2^{-J(Y)} C(Y) (\Deg T + 1) (\Deg \gamma)$ values of $k$.  For each $j$, $Y_j T$ is a polynomial of degree $\le C_1(Y) (\Deg T + 1)$.  We now choose $C(Y) > C_1(Y) 2^{J(Y)}$ so that $ v( Y_j T (z_k) ) = \beta $ holds for more than $(\deg Y_j T) (\deg \gamma)$ values of $k$.

If $\beta_j = 0$, then we see that $Y_j T$ vanishes at $>  \Deg \gamma \Deg Y_j T$ points of $\gamma$.  By B\'ezout's theorem (Theorem \ref{bezoutCurveSurface}), $Y_j T$ vanishes on $\gamma$. If $\beta_j = 1$, then we see that $Y_j T$ fails to vanish at at least one point of $\gamma$.  Since $\gamma$ is irreducible, $Y_j T$ vanishes at only finitely many points of $\gamma$.

Thus at all but finitely many points of $\gamma$, $v (Y_j T) = \beta \in B_Y$.  Hence all but finitely many points of $\gamma$ obey condition $Y$.
\end{proof}

\begin{lemma} \label{construccontag2} Suppose that $Y$ is a constructible condition.   Then there is a constant $C(Y)$ so that the following holds.  Let $T: K^3 \rightarrow K$ be a polynomial.  Suppose that $\gamma_i \in\mathcal{C}\subset\chowVarietyDegD$, and that $T$ obeys $Y$ at almost every point of each $\gamma_i$.  Suppose that all the $\gamma_i$ are contained in an algebraic surface $\BZ(Q)$ for an irreducible polynomial $Q$.  If the number of curves $\gamma_i$ is $\ge C(Y) (\Deg T + 1) \Deg Q$, then $T$ obeys $Y$ at almost every point of $\BZ(Q)$.
\end{lemma}
\begin{proof}
Consider one of the curves $\gamma_i$.  Define $\beta_j(\gamma_i) = 0$ if and only if $Y_j T(x) = 0$ at almost every $x \in \gamma_i$.  Then, for almost every point $x \in \gamma_i$, we have $v (Y_j T(x)) = \beta(\gamma_i)$.  We must have $\beta (\gamma_i) \in B_Y \subset \{ 0, 1 \}^{J(Y)}$.

By pigeonholing, we can find a $\beta \in B_Y$ so that $\beta(\gamma_i) = \beta$ for at least $2^{-J(Y)} C(Y) (\Deg T + 1) \Deg Q$ curves $\gamma_i$.

Suppose that $\beta_j = 0$.  Then $Y_j T$ vanishes on each of these curves $\gamma_i$.  But $\Deg Y_j T \le C(Y) (\Deg T + 1)$.  By choosing $C(Y)$ sufficiently large, the number of curves is greater than $ (\Deg Y_j T)(\Deg Q)$.  Now by a version of B\'ezout's theorem (Theorem \ref{bezoutcurves}), $Y_j T$
and $Q$ must have a common factor.  Since $Q$ is irreducible, we conclude that $Q$ divides $Y_j T$ and so $Y_j T$ vanishes on all of $\BZ(Q)$.

On the other hand, suppose that $\beta_j = 1$.  Then we can find at least one point of $\BZ(Q)$ where $Y_j T$ does not vanish.  Since $Q$ is irreducible, $Y_j T$ vanishes only on a lower-dimensional subvariety of $\BZ(Q)$.

Therefore, at almost every point of $\BZ(Q)$, $v( Y_j T) = \beta \in B_Y$.  Hence, at almost every point of $\BZ(Q)$, $T$ obeys $Y$.
\end{proof}

As a corollary, we see that any constructible condition obeys a version of Proposition \ref{contagiousflecnodes}.
\begin{cor} \label{construccontag3} If $Y$ is a constructible condition, then there is a constant $C(Y)$ so that the following holds.  Suppose that $T \in K[x_1, x_2, x_3]$ is irreducible, and that $\BZ(T)$ contains at least $C(Y) (\Deg T + 1)^2$ curves from $\mathcal{C}\subset\chowVarietyDegD$, each of which contains at least $C(Y) D (\Deg T + 1)$ points where $T$ obeys $Y$.  Then almost every point of $\BZ(T)$ obeys $Y$.
\end{cor}
\begin{proof}
By Lemma \ref{construccontag1}, $T$ obeys $Y$ at almost every point of each of the curves above.  Then by Lemma \ref{construccontag2}, $T$ obeys $Y$ at almost every point of $\BZ(T)$.
\end{proof}

In particular, this result implies Proposition \ref{contagiousflecnodes}, by taking $Y = \Flec_{t, \mathcal{C}, r} \subset K^3 \times K[x]_{\le r}$.

\section{Properties of doubly ruled surfaces}\label{propDoubRuledSurfSec}
In this section we will prove Proposition \ref{implicationsOfDoublyRuled}. For the reader's convenience, we will restate it here.
\begin{implicationsOfDoublyRuledProp}
Let $K$ be an algebraically closed field, let $Z\subset K^3$ be an irreducible surface, and let $\mathcal{C}\subset\chowVarietyDegD$ for some $D\geq 1$. Suppose that $Z$ is doubly ruled by curves from $\mathcal{C}$. Then
\begin{itemize}
 \item $\deg(Z)\leq100 D^2.$
 \item For any $t\geq 1$, we can find two finite sets of curves from $\mathcal{C}$ in $Z$, each of size $t$, so that each curve from the first set intersects each curve from the second set.
\end{itemize}
\end{implicationsOfDoublyRuledProp}
Here is the idea of the proof.  We use the fact that almost every point of $Z$ lies in two curves of $\mc$ contained in $Z$ in order to construct the curves in item (2) above.  Using these curves, we can bound the degree of $Z$ by imitating the proof that an irreducible surface $Z \subset K^3$ which is doubly ruled by lines has degree at most 2.

There is a technical moment in the proof where it helps to know that a generic point of $Z$ lies in only finitely many curves of $\mc$.  This may not be true for $\mc$, but we can find a subset $\mc' \subset \mc$ where it does hold.  In the first section, we explain how to restrict to a good subset of curves $\mc'$.

\subsection{Reduction to the case of finite fibers}

The main tool we will use is the following theorem:
\begin{thm}\label{genericPtSameDim}
Let $Y\subset \FP^{M+N}$ and $W\subset\FP^{N}$ be quasi-projective varieties, let $W$ be irreducible, and let $\pi\colon Y\to W$ be a dominant projection map (or more generally, a dominant regular map). Then there is an open set $O\subset W$ so that the fiber above every point $z\in W$ has dimension $\dim Y-\dim W$ (by convention, a finite but non-empty set has dimension $0$; the empty set has dimension $-1$).
\end{thm}
See e.g.~Theorem 9.9 from \cite{Milne}. In particular, if $Y$ and $W$ are affine varieties then they are quasi-projective, so Theorem \ref{genericPtSameDim} applies.

\begin{cor}\label{dimOfSet}
Let $Y\subset K^{M+N}$ be a constructible set, and let $W\subset K^N$ be an irreducible (affine) variety. Suppose that the image of $\pi\colon Y\to W$ is dense in $W$. Then there is an open set $O\subset W$ so that the fiber above every point $z\in W$ has dimension $\dim Y - \dim W$.
\end{cor}
\begin{proof}
Select affine varieties $Y_1,Y_2\subset K^{M+N}$ so that $Y\subset Y_1\backslash Y_2$ and $\dim Y_2<\dim Y_1$. Note that if the image of $\pi\colon Y\to W$ is dense in $W$, then the image of $\pi\colon Y_1\to W$ must be dense in $W$, i.e. $\pi\colon Y_1\to W$ is a dominant map. Let $O_1\subset Z$ be the open set from Theorem \ref{genericPtSameDim} applied to the map $Y_1\to Z$. If $Y_2\to Z$ is not dominant, then let $O=O_1\backslash \pi(Y_2)$, and we are done.

If $Y_2\to Z$ is dominant, let $O_2\subset Z$ be the open set from Theorem \ref{genericPtSameDim} applied to the map $Y_2\to Z.$ Then for all $z\in O=O_1\cap O_2,$ $\pi_{Y_1}^{-1}(z)$ has dimension $\dim Y_1-\dim W$, and $\pi_{Y_2}^{-1}(z)$ has dimension $\dim Y_2-\dim W.$ Thus $\pi^{-1}(Y)$ has dimension at least $\dim Y_1-\dim W = \dim Y-\dim W$. But this is the maximum possible dimension of a fiber of $\pi_{Y}$ above a point $z\in \mathcal{O}\subset\mathcal{O}_1$. Thus the fiber of $\pi_Y$ above every point $z\in O$ has dimension $\dim Y - \dim W$.
\end{proof}
\begin{lemma}\label{reductionFinFiberLem}
Let $Z$ and $\mathcal{C}$ be as in the statement of Proposition \ref{implicationsOfDoublyRuled}. Define
\begin{equation}\label{zGammaEqn2}
  Y_{Z,\mathcal{C}}=\{(z,\gamma)\in Z\times\mathcal{C}\colon z\in\gamma,\ \gamma\subset Z\},
\end{equation}
and let  $\pi\colon (z,\gamma)\mapsto z$. Then $Y_{Z,\mathcal{C}}$ is a constructible set. Furthermore, there exists a set $\mathcal{C}^\prime\subset \mathcal{C}$ and an open set $O^\prime\subset Z$ so that for every $z\in O^\prime$, the fiber of the projection $\pi\colon Y_{Z,\mathcal{C}}\cap (Z\times \mathcal{C}^\prime)\to Z$ above $z$ has finite cardinality, and this cardinality is $\geq 2$.  The complexity of $\mathcal{C}^\prime$ is at most $O_{C}(1)$, where $C$ is the complexity of $\mathcal{C}$; in fact, $\mathcal{C}^\prime$ is obtained by intersecting $\mathcal{C}$ by the union of two linear spaces.
\end{lemma}
\begin{proof}
First, the set $Y_{Z,\mathcal{C}}$ is constructible. By Theorem \ref{ChevalleyTheorem}, the image of the map $\pi\colon Y_{Z,\mathcal{C}}\to Z$ is constructible. By assumption, it is Zariski dense in $Z$, and on a dense subset, every fiber has cardinality $\geq 2$. If there is an open dense set $O^\prime\subset Z$ where every fiber is finite, then we are done.

Now, suppose that there does not exist an open dense set $O^\prime\subset Z$ where every fiber is finite. Recall that $\mathcal{C}$ is a constructible set, and in particular it is a subset of $K^N$ for some $N\geq 1$. Let $H_1,H_2,\ldots$ be a sequence of linear varieties in $K^N$, with $H_1\supset H_2\supset\ldots,$ and $\codim(H_j)=j$. For each index $j$, let $\tilde H_j = K^3\times H_j$. Then by Corollary \ref{dimOfSet} we have that for each index $j$, the fiber of $\pi_j\colon(Y_{Z,\mathcal{C}}\cap H_j)\to Z$ above a generic point of $Z$ has dimension $\dim(Y_{Z,\mathcal{C}}\cap H_j)-\dim Z$ if $\dim(Y_{Z,\mathcal{C}}\cap H_j)-\dim Z\geq 0$, and the fiber is empty otherwise. When $j=0$, this quantity is $\geq 1$ by assumption. On the other hand, when $j=N$, then $Y_{Z,\mathcal{C}}\cap H_j=\emptyset$, so the fiber above a generic point of $\pi_N$ is empty and thus has dimension $-1$. Furthermore, since $K$ is algebraically closed,
\begin{equation}
 \dim(Y_{Z,\mathcal{C}}\cap H_{j})-1\leq \dim(Y_{Z,\mathcal{C}}\cap H_{j+1})\leq \dim(Y_{Z,\mathcal{C}}\cap H_j).
\end{equation}
Thus there is an index $j_0$ so that the fiber of $\pi_{j_0}$ above a generic point of $Z$ is finite and non-empty, and the fiber of $\pi_{j_0+1}$ above a generic point of $Z$ is empty.

We can repeat this procedure with a second collection $\{\tilde H_j^\prime\}$ of hyperplanes so that no variety $\tilde H_j^\prime$ contains any irreducible component of $\pi^{-1}_{j_0}(Z)$. Arguing as above, we obtain a second index $j_0^\prime$ so that the fiber of $\pi_{j_0^\prime}$ above a generic point of $Z$ is finite and non-empty. On the other hand, the fibers of $\pi_{j_0}$ and $\pi_{j_0^\prime}$ above a generic point of $Z$ are disjoint. Thus the fiber of $\pi\colon (Y_{Z,\mathcal{C}}\cap (\tilde H_{j_0}\cup \tilde H_{j_0^\prime}^\prime)\to Z$ above a generic point of $z$ is finite and has cardinality $\geq 2$.
\end{proof}

\begin{lemma}\label{dimensionOfCPrime}
Let $Z$ be as in the statement of Proposition \ref{implicationsOfDoublyRuled} and let $\mathcal{C}^\prime$ be as in Lemma \ref{reductionFinFiberLem}. Then $Y_{Z,\mathcal{C}}\cap (Z\times \mathcal{C}^\prime)$ has dimension 2, and the image of $Y_{Z,\mathcal{C}}\cap (Z\times \mathcal{C}^\prime)$ under the projection $(z,\gamma)\mapsto \gamma$ has dimension $1$.
\end{lemma}
\begin{proof}
 The first statement follows from Theorem \ref{genericPtSameDim}. The second statement follows from the observation that for a generic $\gamma\in\mathcal{C}^\prime$, the fiber above the projection $(x,\gamma)\mapsto\gamma$ has dimension 1.
\end{proof}

Throughout the rest of this section, we will fix the set $\mathcal{C}^\prime$ and $O^\prime$.

The following subsets of $\mathcal{C}^\prime$ play an important role in the argument.

\begin{defn}
Let $X\subset K^3$ be a constructible set. Define
 \begin{equation*}
  \mathcal{C}^X\defeq\{\gamma\in\mathcal{C}^\prime\colon \gamma\cap X\neq\emptyset\}
 \end{equation*}
 and
  \begin{equation*}
  \mathcal{C}_X \defeq \{\gamma\in\mathcal{C}^\prime\colon \gamma\subset X\}.
 \end{equation*}
\end{defn}

The notation here is potentially confusing, so we reiterate that these are subsets of $\mathcal{C}^\prime$.  All the curves we consider in the rest of this Section belong to $\mathcal{C}^\prime$.  We are leaving the prime out of the notation just because it is awkward to have to write $(\mathcal{C}^\prime)^X$ many times.

The sets $\mathcal{C}^X$ and $ \mathcal{C}_X$ are constructible sets.

\subsection{Constructible families of curves}

\begin{lemma}\label{complexityOfUnionOverConstructibleSet}
 If $\mathcal{C}^{\prime\prime}\subset \mathcal{C}^\prime$ is a constructible set of complexity $\leq C$, then
 \begin{equation*}
 U(\mathcal{C}^{\prime\prime})\defeq \bigcup_{\gamma\in \mathcal{C}^{\prime\prime}}\gamma
 \end{equation*}
 is a constructible set of complexity $O_{C_0,D,C}(1)$.
\end{lemma}

\begin{proof} The union $U (\mc'')$ is the projection of $Y_{Z,\mathcal{C}}\cap (Z\times \mc'') \subset Z \times \mc''$ to the $Z$ factor.  Since $Y_{Z, \mc}$ and $\mc''$ are constructible, $U(\mc'')$ is constructible too.
\end{proof}

\begin{lemma}[Selecting a curve from a dense family]\label{denseCurveSelection}
Let $Z\subset K^3$ be an irreducible surface, and let $\mathcal{C}^{\prime\prime}\subset\mathcal{C}_Z$ be an infinite set of curves. Let $X\subset Z$ be a dense, constructible set. Then there exists a curve $\gamma\in \mathcal{C}^{\prime\prime}$ so that $\gamma\cap X$ contains all but finitely many points of  $\gamma$.
\end{lemma}

\begin{proof} We note that a constructible subset of $Z$ is either contained in a 1-dimensional subset of $Z$ or else contains a dense open set $O \subset Z$.  Since $X$ is a dense constructible subset of $Z$, there is a finite list of irreducible curves $\beta_j \subset Z$ so that $X \supset Z \setminus \cup_j \beta_j$.  Since $\mathcal{C}^{\prime\prime}$ is infinite, we can choose $\gamma \in \mathcal{C}^{\prime\prime}$ with $\gamma$ not equal to any of the curves $\beta_j$.   Therefore, $\gamma \cap \beta_j$ is finite for each $j$, and so all but finitely many points of $\gamma$ lie in $X$.
\end{proof}

\subsection{Constructing many intersecting curves}

\begin{lemma}\label{findLargeFamilyCurvesLem}
Let $Z,$ $\mathcal{C}^\prime$, and $O^\prime$ be as above.  Then we can construct an infinite sequence of curves $\gamma_1, \gamma_2, ...$ in $\mathcal{C}_Z$ so that for any $\ell \ge 1$,

\begin{enumerate}
\item $\mathcal{C}_Z \cap \mathcal{C}^{\gamma_1 \cap O'} \cap ... \cap \mathcal{C}^{\gamma_\ell \cap O'}$ is infinite.
\item At most two curves from $\{\gamma_1,\ldots,\gamma_\ell\}$ pass through any point $z\in O'$.
\end{enumerate}
\end{lemma}

\begin{proof}
We will prove the theorem by induction on $\ell$.  We begin with the case $\ell = 1$.  By Lemma \ref{denseCurveSelection}, we can choose $\gamma_1 \in \mc_Z$ so that $\gamma_1 \cap O'$ is dense in $\gamma_1$.  Each point $z \in O'$ lies in at least two curves of $\mc_Z$, so each point of $\gamma_1 \cap O'$ lies in a curve of $\mc_Z$ besides $\gamma_1$.  Therefore, $\mathcal{C}_Z \cap \mathcal{C}^{\gamma_1 \cap O'}$ is infinite.  This checks Property (1) above, and Property (2) is vacuous in the case $\ell = 1$.

Now we do the inductive step of the proof.  Suppose that we have $\gamma_1, ..., \gamma_{\ell}$ with the desired properties.  We have to find $\gamma_{\ell+1}$.

We define

$$\mc_\ell \defeq \mathcal{C}_Z \cap \mathcal{C}^{\gamma_1 \cap O'} \cap ... \cap \mathcal{C}^{\gamma_l \cap O'}.$$

Next we define a finite set of undesirable curves $B_\ell$.  As a warmup, we define $D_\ell$ to be the set of intersection points of $\gamma_1, ..., \gamma_\ell$ in $O'$:

$$ D_\ell \defeq \{ z \in O' \colon z \textrm{ lies in at least two of the curves } \gamma_1, ..., \gamma_\ell \}. $$

The set $B_{\ell}$ is the union of the curves $\{ \gamma_i \}_{i=1}^\ell$ together with the union of all the curves of $\mc_Z$ that pass through a point of $D_\ell$:

$$ B_{\ell} \defeq \left( \cup_{i=1}^l \gamma_i \right) \bigcup \left( \cup_{z \in D_\ell} \mc_Z \cap \mc^z \right). $$

The set $D_\ell$ is finite because any two irreducible curves can intersect in only finitely many points.  For each $z \in O'$, $\mc_Z \cap \mc^z$ is finite, and so $B_\ell$ is finite.  We will choose $\gamma_{\ell+1} \notin B_\ell$.  This will guarantee that $\gamma_{\ell+1}$ is distinct from the previous curves, and it will also guarantee Property (2) above.

Our process depends on whether $\mc_Z \setminus \mc_\ell$ is finite or infinite.

Suppose $\mc_Z \setminus \mc_\ell$ is infinite.  By Lemma \ref{complexityOfUnionOverConstructibleSet}, we know that $U (\mc_\ell)$ is a constructible subset of $Z$.  By Property (1), we know that $\mc_\ell$ is infinite, and so $U(\mc_\ell)$ must be a dense constructible set in $Z$.  Therefore, $U(\mc_\ell) \cap O'$ is also dense and constructible.  Since $\mc_Z \setminus \mc_\ell$ is infinite, we can use Lemma \ref{denseCurveSelection} to choose $\gamma_{\ell+1}$ in $\mc_Z \setminus (\mc_\ell \cup B_\ell)$ so that

$$ |\gamma_{\ell+1} \cap O' \cap U(\mc_\ell) | = \infty. $$

Since $\gamma_{\ell+1} \notin \mc_\ell$, we see that there are infinitely many curves of $\mc_\ell$ that intersect $\gamma_{\ell +1} \cap O'$.  This establishes Property (1), finishing the case that $\mc_Z \setminus \mc_\ell$ is infinite.

Suppose instead that $\mc_Z \setminus \mc_\ell$ is finite.  Then $O' \setminus U(\mc_Z \setminus \mc_\ell)$ is a dense, constructible subset of $Z$.  Since $\mc_\ell$ is infinite,
we can use Lemma \ref{denseCurveSelection} to choose $\gamma_{\ell+1}$ in $\mc_\ell \setminus B_\ell$ so that

$$ |\gamma_{\ell+1} \cap O' \setminus U(\mc_Z \setminus \mc_\ell) | = \infty. $$

If $z \in O'$, then $z$ lies in two distinct curves of $\mc_Z$.  If $z \in O' \setminus U(\mc_Z \setminus \mc_\ell)$, then $z$ lies in two distinct curves of $\mc_\ell$.  One of these curves may be $\gamma_{l+1}$, but one of them must be a different curve.  Therefore there are infinitely many curves of $\mc_\ell$ that intersect $\gamma_{\ell+1} \cap O' $.  This establishes Property (1) and finishes the induction.
\end{proof}

We can now give the proof of Proposition \ref{implicationsOfDoublyRuled}.

\begin{proof} Let $\gamma_i$ be the curves in Lemma \ref{findLargeFamilyCurvesLem}.
Let $N$ be a parameter at our disposal.  Let $X$ be a set of $N^2$ points, with $N$ points on each curve $\gamma_i$ for $i = 1, ..., N$.  Let $Q$ be a minimal degree polynomial that vanishes on the set $X$.

Since $\dim \Poly_D(K^3) \ge (1/6) D^3$, we have $\deg Q \le 3 |X|^{1/3} = 3 N^{2/3}$.  If $D \cdot 3 N^{2/3} < N$, then the B\'ezout theorem (Theorem \ref{bezoutCurveSurface}) implies that $Q$ vanishes on $\gamma_i$ for each $i = 1, ..., N$.  Suppose from now on that $N > 27 D^3$, which guarantees that $Q$ indeed vanishes on $\gamma_i$ for each $i = 1, ..., N$.

Next, we recall that there are infinitely many curves in $\mathcal{C}_Z \cap \mathcal{C}^{\gamma_1 \cap O'} \cap ... \cap \mathcal{C}^{\gamma_N \cap O'}$.  Let $\gamma'$ be any one of these curves.  We recall that any point of $O'$ lies in at most two of the curves $\gamma_i$, and so $\gamma'$ intersects the curves $\gamma_i$ at at least $N/2$ distinct points.  So $Q$ vanishes at at least $N/2$ points of $\gamma'$.  If $D (\deg Q) < N/2$, then $Q$ must vanish at every point of $\gamma'$.  Now $\deg Q \le 3 N^{2/3}$ and so it suffices to check that $D \cdot 3 N^{2/3} < N/2$.  We now suppose that $N > 6^3 D^3$, which guarantees that $Q$ vanishes on all the infinitely many curves $\gamma ' \in
\mathcal{C}_Z \cap \mathcal{C}^{\gamma_1 \cap O'} \cap ... \cap \mathcal{C}^{\gamma_N \cap O'}$.
At this point, we can choose $N = 6^3 D^3 + 1$.

Suppose that our surface $Z$ is the zero set of an irreducible polynomial $T$.  We now see that $Z \cap Z(Q) = Z(T) \cap Z(Q)$ contains infinitely many distinct curves.  By B\'ezout's theorem, Theorem \ref{bezoutcurves}, it follows that $Q$ and $T$ must have a common factor.  But $T$ is irreducible, so $T$ must divide $Q$.  But then the degree of $T$ is at most the degree of $Q$, which is at most $3 N^{2/3}$ which is at most $200 D^2$.  This proves the desired bound on the degree of $
Z$.

The second item from Proposition \ref{implicationsOfDoublyRuled} follows immediately from Lemma \ref{findLargeFamilyCurvesLem}.  For any $\ell \ge 1$, by Lemma \ref{findLargeFamilyCurvesLem} there are infinitely many curves of $\mc$ that intersect each of the curves $\{\gamma_1,\ldots,\gamma_\ell\}$.
\end{proof}

\section{Proving Theorem \ref{mainStructureThm}}\label{puttingItAllTogetherSection}
%
% %
We begin with a corollary to Corollary \ref{contagiousflecnodes} and Theorem \ref{tangentImpliesTrappedProp}:
\begin{cor}\label{rdtImpliesRdInfty}
Fix $D\geq 1$ and $C$. Then there exists a number $r$ (large) and $c_1$ (small) so that for every constructible set $\mathcal{C}\subset \chowVarietyDegD$ of complexity at most $C$ and every irreducible polynomial $T\in K[x_1,x_2,x_3]$ with $\deg T\leq c_1\operatorname{char}(K)$, there exists a (Zariski) open subset $O\subset \BZ(T)$ so that the following holds. If $x\in O$ is $(t,\mathcal{C},r)$-flecnodal for $T$, then there exist (at least) $t$ curves in $\mathcal{C}$ that contain $x$ and are contained in $\BZ(T)$.
 \end{cor}
\begin{lemma}\label{popCurveLem}
Fix $D>0, C>0$. Then there are constants $c_2,C_3,C_4$ so that the following holds. Let $k$ be a field and let $K$ be the algebraic closure of $K$. Let $\mathcal{C}\subset\chowVarietyDegD$ be a constructible set of complexity at most $C$. Let $\finiteSetCurves$ be a collection of $n$ irreducible algebraic curves in $k^3$ whose algebraic closures are elements of $\mathcal{C}$. Suppose furthermore that $\operatorname{char}(k)=0$ or $n\leq c_2 (\operatorname{char}(k))^2$.

Let $A>C_3n^{1/2}$.  Suppose that for each $\gamma\in\finiteSetCurves$, there are $\geq A$ points $z\in\gamma$ that are incident to some curve from $\finiteSetCurves$ distinct from $\gamma$.
 Then there exists an irreducible surface $Z\subset k^3$ with the following properties:
 \begin{itemize}
 \item $Z$ contains at least $A/C_4$ curves from $\finiteSetCurves$.
 \item $Z$ is ``doubly ruled'' by curves from $\mathcal{C}$ in the sense of Definition \ref{defnOfDoublyRuled} (and hence has degree at most $100 D^2$ by Proposition \ref{implicationsOfDoublyRuled}).
\end{itemize}
\end{lemma}
Before proving Lemma \ref{popCurveLem}, we will show how it implies Theorem  \ref{mainStructureThm}. For the reader's convenience, we will recall the theorem here.
\begin{mainStructureThmThm}
Fix $D>0, C>0$. Then there are constants $c_1,C_1,C_2$ so that the following holds. Let $k$ be a field and let $K$ be the algebraic closure of $K$. Let $\mathcal{C}\subset\chowVarietyDegD$ be a constructible set of complexity at most $C$. Let $\finiteSetCurves$ be a collection of $n$ irreducible algebraic curves in $k^3$, with $\finiteSetCurves\subset\mathcal{C}$ (see Definition \ref{AcontainedC}). Suppose furthermore that $\operatorname{char}(k)=0$ or $n\leq c_1 (\operatorname{char}(k))^2$.

Then for each number $A>C_1n^{1/2}$, at least one of the following two things must occur:
\begin{itemize}
 \item There are at most $C_2An$ points in $k^3$ that are incident to two or more curves from $\finiteSetCurves$.
 \item There is an irreducible surface $Z\subset k^3$ that contains at least $A$ curves from $\finiteSetCurves$. Furthermore, $\hat Z$ is doubly ruled by curves from $\mathcal{C}$. See Definition \ref{defnOfDoublyRuled} for the definition of doubly ruled, and see Proposition \ref{implicationsOfDoublyRuled} for the implications of this statement.
\end{itemize}
\end{mainStructureThmThm}

\begin{proof}[Proof of Theorem  \ref{mainStructureThm} using Lemma \ref{popCurveLem}]
\begin{defn}
Let $\finiteSetCurves$ be a collection of curves in $k^3$. Define
\begin{equation*}
\mathcal{P}_2(\finiteSetCurves)=|\{x\in k^3\colon x\ \textrm{is incident to at least two curves from}\ \finiteSetCurves\}|.
\end{equation*}
\end{defn}
Let $D$ and $C$ be as in the statement of Theorem \ref{mainStructureThm}. Fix a field $k$, a constructible set $\mathcal{C}\subset\mathcal{C}_{3,D}$ (of complexity at most $C$), and a number $A$. We will prove the theorem by induction on $n$, for all $n\leq \min\Big(C_1^{-2}A^2, c_1(\operatorname{char} K)^2\Big)$. The case $n=1$ is immediate. Now, suppose the statement has been proved for all collections of curves in $\mathcal{C}$ of size at most $n-1$.

Applying Lemma \ref{popCurveLem} (with the value $A^\prime=C_4A$), we conclude that either there is an irreducible surface $Z$ that is doubly ruled by curves from $\mathcal{C}$ and that contains at least $A^\prime/C_4=A$ curves from $\finiteSetCurves,$ or there is a curve $\gamma\in\finiteSetCurves$ so that $|\pts_2(\finiteSetCurves)\cap\gamma|<C_4A$. If the former occurs then we are done.

If the latter occurs, then let $\finiteSetCurves^\prime=\finiteSetCurves\backslash\{\gamma_0\}$. Then $|\finiteSetCurves^\prime|=n-1$, so the collection $\finiteSetCurves^\prime$ satisfies the induction hypothesis. Thus if we select $C_2\geq C_4$, we have
\begin{equation}
\begin{split}
 \mathcal{P}_2(\finiteSetCurves)&< \mathcal{P}_2(\finiteSetCurves\backslash\{\gamma_0\})+C_4A\\
 &\leq C_2 A(n-1)+ C_4 A \\
 &\leq C_2An.
\end{split}
 \end{equation}
This closes the induction and establishes Theorem  \ref{mainStructureThm}.
\end{proof}

\begin{proof}[Proof of Lemma \ref{popCurveLem}]
\begin{prop}[Degree reduction]\label{degreeReduction}
For every $D\geq 1$, there are constants $C_5,C_6$ so that the following holds. Let $\finiteSetCurves$ be a collection of $n$ irreducible curves of degree $\leq D$ in  $k^3$, and let $A\geq C_5n^{1/2}$. Suppose that for each $\gamma\in\finiteSetCurves$, there are $\geq A$ points $z\in\gamma$ that are incident to some curve from $\finiteSetCurves$ distinct from $\gamma$. Then there is a polynomial $P$ of degree at most $C_6n/A$ whose zero-set contains every curve from $\finiteSetCurves$.
\end{prop}
We will prove Proposition \ref{degreeReduction} in Appendix \ref{degreeReductionSec}.

Now, factor $P=P_1\ldots P_{\ell}$ into irreducible components. For $j=1,\ldots,\ell$, define
\begin{equation*}
 \finiteSetCurves_j=\{\ell\in\finiteSetCurves\colon \ell\subset \BZ(P_j),\ \ell\not\subset Z_i\ \textrm{for any}\ i<j\}.
\end{equation*}
Note that for each index $j$ and each curve $\gamma\in \finiteSetCurves_j$,
\begin{equation}
|\{p\in k^3\ p\in \gamma\cap\gamma^\prime,\ \textrm{for some}\ \gamma^\prime\in\finiteSetCurves_i,\ i\neq j\}|<\deg P<A/2,
\end{equation}
provided $A>(2C_6n)^{1/2}$. Thus each curve $\gamma$ is incident to at least $A/2$ other curves $\gamma^\prime$ that lie in the same set $\finiteSetCurves_j$ (and are therefore contained in the same surface $Z_j$).

By pigeonholing, exists an index $j$ with
\begin{equation}\label{pigeonHole2}
 |\finiteSetCurves_j| \geq \frac{A}{2C_6}
\end{equation}
and
\begin{equation}\label{pigeonHole1}
\begin{split}
 |\finiteSetCurves_j| & \geq \frac{1}{2}\frac{n}{(\deg P)^2}(\deg Z_j)^2\\
 &\geq\frac{1}{2}\frac{n}{(C_6n/A)^2}(\deg Z_j)^2\\
  &\geq\Big(\frac{A^2}{2C_6^2n}\Big)(\deg Z_j)^2.
 \end{split}
\end{equation}
Select $A_4$ (from the statement of Lemma \ref{popCurveLem}) to be larger than $2C_6$, and let $Z_0$ be this irreducible component.

By Lemma \ref{DIsLargeWhenCommongComponent}, for each curve $\gamma\in\finiteSetCurves_j,$ there are at least $A/2$ points on $\gamma$ that are $(2,\mathcal{C},r)$-flecnodal. Thus by Proposition \ref{contagiousflecnodes}, for each $r>0$, there is a Zariski open set $O_r\subset Z_0$ consisting of $(2,\mathcal{C},r)$-flecnodal points. If we select $r$ sufficiently large (depending on $D$, where $\mathcal{C}\subset\mathcal{C}_{3,D}$) and if $\frac{A^2}{2C_6^2n}$ is sufficiently large (depending on $r$) (this can be guaranteed if we select $C_3$ from the statement of Lemma \ref{popCurveLem} to be sufficiently large depending on $D$), then by Corollary \ref{rdtImpliesRdInfty}, there exists a Zariski open set $O\subset Z_0$ so that for every point $x\in O$, there are two curves from $\mathcal{C}$ passing through $x$ contained in $Z$.  In other words, $Z_0$ is doubly ruled by curves from $\mc$, as in Definition
\ref{defnOfDoublyRuled}.
\end{proof}

\appendix
\section{A quantitative ascending chain condition}\label{AQCCApp}
In this section we will prove Proposition \ref{quantACCProp}. For the reader's convenience, we re-state it here:
\begin{quantACCPropProp}
Let $\tilde K$ be a field, let $N\geq 0,$ and let $\tau\colon \NN\to\NN$ be a function. Then there exists a number $M_0$ with the following property. Let $\{I_i\}$ be a sequence of ideals in $\tilde K[x_1,\ldots,x_N]$, with $\complexity(I_i)\leq \tau(i).$ Then there exists a number $r_0\leq M_0$ so that $I_{r_0}\subset I_1+ ... + I_{r_0-1}$.
\end{quantACCPropProp}

\subsection{Reverse lexicographic  order}
\begin{defn}
Given two $(N+1)$--tuples $\bell=(\ell_0,\ldots,\ell_N),$ $\bell^\prime =(\ell_0^\prime,\ldots,\ell_N^\prime)$, we say $\bell \prec\bell^\prime$ if $\bell\neq\bell^\prime$, and one of the following holds
\begin{itemize}
\item $\ell_N<\ell_N^\prime,$
\item $\ell_N=\ell_N^\prime$ and $\ell_{N-1}<\ell_{N-1}^\prime$,\\ $\vdots$
\item $\ell_N=\ell_N^\prime,\ell_{N-1}=\ell_{N-1}^\prime,\ldots,\ell_1=\ell_1^\prime$, and $\ell_0<\ell_0^\prime$.
\end{itemize}
We will only use $\prec$ to compare two tuples of the same length. The relation $\prec$ is transitive.
\end{defn}
\begin{defn}
If $\bell$ is a tuple, we define $|\bell|=|\ell_0|+\ldots+|\ell_{N}|$. In our applications, the entries will always be non-negative. We will use $\mathbf{0}$ to denote the tuples whose entries are all 0s (the length of the tuple should be apparent from context).
\end{defn}
\begin{lemma}[length of chains]\label{lengthOfChainLem}
Let $N\geq 0$ and let $\tau\colon\NN\to\NN$ (in our applications, we will have something like $N=3,\ \tau(t)=100t^3$). Then there exists a number $M_0$ with the following property. Let $\{\bell_i\}$ be a sequence of $(N+1)$--tuples of non-negative integers. Suppose that the sequence is weakly monotonically decreasing under the $\prec$ order. Suppose furthermore that for each index $i$, $|\bell_i|\leq \tau(i)$. Then there exists some $r_0\leq M_0$ so that $\bell_{r_0-1}=\bell_{r_0}$.
\end{lemma}

\subsection{Hilbert functions and Hilbert polynomials}
Let $\tilde K$ be a field, and let $I\subset \tilde K[x_1,\ldots,x_N]$ be an ideal. We define $I_{\leq t}$ to be the set of all polynomials in $I$ that have degree at most $t$; this set has the structure of a $\tilde K$--vector space. We define the Hilbert function
\begin{equation}
H_I(t) = \dim_{\tilde K}(\tilde K[x_1,\ldots,x_N]_{\leq t}/I_{\leq t}).
\end{equation}
\begin{thm}[Hilbert]
There exists a polynomial $HP_I\in\RR[t]$ so that for all $t\in\NN$ sufficiently large, $HP_I(t)=H_I(t)$. Furthermore, $HP_I(t)$ is an integer for all $t\in\NN$.
\end{thm}
\begin{defn}
If $I\subset \tilde K[x_1,\ldots,x_N]$, let $\bell_I=(\ell_0,\ldots,\ell_N)$, where $\ell_j=j!\operatorname{coeff}(HP_I,j)$. Here $\operatorname{coeff}(HP_I,j)=\frac{1}{j!}HP_I^{(j)}(0)$ is the coefficient of $t^j$ in the polynomial $HP_I$.
\end{defn}
\begin{lemma}
Let $I$ be an ideal. Then $\bell_I$ is a tuple of non-negative integers.
\end{lemma}
\begin{proof}
This follows from the Maucaulay representation of a Hilbert polynomial (see i.e.~\cite[Prop 1.3]{ChardinSocias} for further details).
\end{proof}
\begin{prop}
If $I\subset I^\prime$, then $\bell_{I}^\prime \preceq\bell_{I}$. If furthermore $\bell_{I}=\bell_{I}^\prime$, then $I=I^\prime$.
\end{prop}
\begin{proof}
If $I\subset I^\prime,$ then $H_I(t)\leq H_{I^\prime}(t)$, and this establishes the first statement. On the other hand, if $\bell_{I}=\bell_{I}^\prime$ then $H_I(t)= H_{I^\prime}(t)$ for all sufficiently large $t$, and this immediately implies $I=I^\prime$.
\end{proof}
\begin{lemma}[Quantitative bounds on coefficients of Hilbert Polynomials]\label{quantBdHilbertPoly}
Let $I\subset \tilde K[x_1,\ldots,x_N]$. Then $|\bell_I|$ is bounded by a function that depends only on $N$ and $\complexity(I)$. i.e.~the sum of the coefficients of the Hilbert polynomial of the ideal $(f_1,\ldots,f_\ell)$ is controlled by $\ell$ and the maximal degree of $f_1,\ldots,f_\ell$.
\end{lemma}
%see e.g.~\cite{Schenck}. Alternately,

%\COMMENTTOJOSH{get a copy of \cite{Schenck}. Verify this and cite a specific place in \cite{Schenck}.}

%
%
We can now prove Proposition \ref{quantACCProp}. Let $\tilde I_j = (I_1+\ldots+I_j)$, so $\tilde I_j\subset \tilde I_{j+1}$ for each index $j$. By Lemma \ref{quantBdHilbertPoly}, there is a function $\tilde \tau_j$ (depending only on $N$ and $\tau$) so that $|\bell_{\tilde I_j}|\leq\tilde\tau(j)$. Thus by Lemma \ref{lengthOfChainLem} applied to $\tilde\tau$, there is a number $M_0$ (depending only on $N$ and $\tau$) so that $\bell_{\tilde I_{r_0-1}}=\bell_{\tilde I_{r_0}}$ for some $r_0\leq M_0$. We conclude that $\tilde I_{r_0-1}=\tilde I_{r_0}$ and thus $I_{r_0}\subset(I_1 + ... + I_{r_0-1})$.
\section{Degree reduction}\label{degreeReductionSec}
In this section we will prove Proposition \ref{degreeReduction}. The proof is similar to arguments found in \cite{GuthKatz2}.

We will require several Chernoff-type bounds for sums of Bernoulli random variables. For convenience, we will gather them all here.
\begin{thm}[Chernoff]
Let $X_1,\ldots,X_N$ be iid Bernoulli random variables with $\mathbf{P}(X_i=1)=p,\ \mathbf{P}(X_i=0)=1-p$. Then
\begin{align*}
&\mathbf{P}\Big(\frac{1}{N}\sum_{i=1}^N X_i \leq p-\varepsilon\Big)\leq \bigg(\Big(\frac{p}{p-\varepsilon}\Big)^{p-\varepsilon}\Big(\frac{1-p}{1-p+\epsilon}\Big)^{1-p+\varepsilon}\bigg)^N,\\
&\mathbf{P}\Big(\frac{1}{N}\sum_{i=1}^N X_i \geq p+\varepsilon\Big)\leq \bigg(\Big(\frac{p}{p+\varepsilon}\Big)^{p+\varepsilon}\Big(\frac{1-p}{1-p-\epsilon}\Big)^{1-p-\varepsilon}\bigg)^N.
\end{align*}
\end{thm}
\begin{cor}\label{BernouliCor}
Let $X_1,\ldots,X_N$ be iid Bernoulli random variables with $\mathbf{P}(X_i=1)=p,\ \mathbf{P}(X_i=0)=1-p$. Suppose $p\geq N^{-1}$. Then
\begin{align}
&\mathbf{P}\Big(\sum_{i=1}^N X_i \leq \frac{pn}{100}\Big)\leq 1\!/2,\label{smallSumBd}\\
&\mathbf{P}\Big(\sum_{i=1}^N X_i \geq 100pn \Big)\leq 1\!/4.\label{largeSumBd}
\end{align}
\end{cor}

\begin{prop}\label{prob12BernouliProp}
Let $X_1,\ldots,X_N$ be iid Bernoulli random variables with $\mathbf{P}(X_i=1)=\mathbf{P}(X_i=0)=1/2$. Suppose $N\geq 100$. Then
\begin{equation}
\mathbf{P}\Big(\sum X_i < \frac{99}{100} \frac{N}{2}\Big)<\frac{1}{4}.
\end{equation}
\end{prop}

\begin{prop}[Polynomial interpolation]\label{polyInterpLemma}
Let $\mathcal{L}_1$ be a collection of $n$ irreducible degree curves of degree $\leq D$ in $k^3$. Then there is a polynomial $P\in k[x_1,x_2,x_3]$ of degree at most $100Dn^{1/2}$ that contains all of the curves in $\mathcal{L}_1$.
\end{prop}

We are now ready to prove Proposition \ref{prob12BernouliProp}. For the readers convenience we will re-state it here.
\begin{degreeReductionProp}
For every $D\geq 1$, there are constants $C_0,C_1$ so that the following holds. Let $\finiteSetCurves$ be a collection of $n$ irreducible curves of degree $\leq D$ in $k^3$, and let $A\geq C_0n^{1/2}$. Suppose that for each $\gamma\in\finiteSetCurves$, there are $\geq A$ points $z\in\gamma$ that are incident to some curve from $\finiteSetCurves$ distinct from $\gamma$. Then there is a polynomial $P$ of degree at most $C_1n/A$ whose zero-set contains every curve from $\finiteSetCurves$.
\end{degreeReductionProp}
\begin{proof}
For each $D$ we will prove the result by induction on $n$. The case $n\leq 10^3$ follows from Proposition \ref{prob12BernouliProp}, provided we take $C_1\geq 10^{5/2}D$. Now assume the result has been proved for all sets $\tilde{\finiteSetCurves}$ of size at most $n-1$.

For each curve $\gamma\in\finiteSetCurves$, choose a set $\pts_\gamma\subset \pts_2(\finiteSetCurves)$ of size $A$. Each point in $\pts_\gamma$ is hit by at least one curve from $\finiteSetCurves$. Furthermore, no curve from $\finiteSetCurves$ can intersect $\gamma$ in more than $D^2$ points. Thus we can select a set $\pts_\gamma^\prime$ of size $A/D^2$ and a collection  $\finiteSetCurves_{\gamma}\subset\finiteSetCurves$ of size $A/D^2$ so that each curve is incident to $\gamma$ at exactly one point of $\pts_\gamma^\prime$, and no two curves from $\finiteSetCurves_{\gamma}$ are incident to $\gamma$ at the same point of $\pts_\gamma^\prime$.

Let $p=C_2n/A^2$, where $C_2=C_2(D)$ is a constant to be chosen later. Let $\finiteSetCurves^\prime\subset\finiteSetCurves$ be a subset of $\finiteSetCurves$ obtained by choosing each curve in $\finiteSetCurves$ with probability $p$. By \eqref{largeSumBd} from Corollary \ref{BernouliCor}, we have
\begin{equation}
\mathbf{P}\Big(|\finiteSetCurves^\prime|>100p|\finiteSetCurves|\Big)<1\!/4.
\end{equation}

By \eqref{smallSumBd} from Corollary \ref{BernouliCor}, for each $\gamma\in\finiteSetCurves$ we have
\begin{equation*}
\mathbf{P}\Big(|\finiteSetCurves_\gamma\cap\finiteSetCurves^\prime|<\frac{p|\finiteSetCurves_\gamma|}{100} \Big)<1\!/2.
\end{equation*}
Since the above events are independent, by Proposition \ref{prob12BernouliProp} we have
\begin{equation*}
\mathbf{P}\Big(\big|\big\{\gamma\in\finiteSetCurves\colon\ |\finiteSetCurves_\gamma\cap\finiteSetCurves^\prime|<\frac{p|\finiteSetCurves_\gamma|}{100}\big\}\big|<\frac{99}{200}|\finiteSetCurves|\Big)<1\!/4.
\end{equation*}

Thus, we can select a set $\finiteSetCurves^\prime\subset\finiteSetCurves$ so that
\begin{equation*}
|\finiteSetCurves^\prime|\leq 100p|\finiteSetCurves|,
\end{equation*}
and
\begin{equation}\label{manyCurvesHitAPrimeMultiply}
\big|\big\{\gamma\in\finiteSetCurves\colon\ |\finiteSetCurves_\gamma\cap\finiteSetCurves^\prime|>\frac{p|\finiteSetCurves_\gamma|}{100}\big\}\big|>\frac{99}{200}|\finiteSetCurves|.
\end{equation}

Using Proposition \ref{polyInterpLemma}, we can find a polynomial $P_1\in k[x_1,x_2,x_3]$ of degree $\leq 100D (100p|\finiteSetCurves|)^{1/2}$ that contains every line from $\finiteSetCurves^\prime$. If $C_2=C_2(D)$ is chosen sufficiently large, then
\begin{equation*}
D(\deg P)+1<\frac{p|\finiteSetCurves_\gamma|}{100}.
\end{equation*}
Thus if $|\finiteSetCurves_\gamma\cap\finiteSetCurves^\prime|>\frac{p|\finiteSetCurves_\gamma|}{100}$ then $\gamma\subset \BZ(P)$. Let $\finiteSetCurves_1\defeq\{\gamma\in\finiteSetCurves\colon\gamma\subset \BZ(P)\}$. Let $\tilde{\finiteSetCurves}\defeq\finiteSetCurves\backslash\finiteSetCurves_1$.

By \eqref{manyCurvesHitAPrimeMultiply}, $|\finiteSetCurves_1|\geq\frac{99}{200}|\finiteSetCurves|$, and thus $|\tilde{\finiteSetCurves}|\leq\frac{101}{200}|\finiteSetCurves|$. If $\gamma\in\tilde{\finiteSetCurves}$ then $\gamma$ can intersect $\BZ(P_1)$ in at most $D(\deg P_1)$ places. This implies
\begin{equation*}
|\gamma\cap \pts_2(\finiteSetCurves_1)|<D(\deg P)+1<\frac{1}{100}A,
\end{equation*}
provided $C_2=C_2(D)$ is chosen sufficiently small depending on $D$.

But recall that $|\gamma\cap \pts_2(\finiteSetCurves)|\geq A$. This means that for each curve  $\gamma\in \tilde{\finiteSetCurves}$,
\begin{equation*}
|\gamma\cap\pts_2(\tilde{\finiteSetCurves})|\geq\frac{99}{100}A.
\end{equation*}
Since $|\tilde{\finiteSetCurves}|\leq\frac{101}{200}|\finiteSetCurves|$, we have
\begin{equation}
\begin{split}
\frac{99}{100}A & \geq \frac{99}{100}C_0n^{1/2} \\
&\geq \frac{99}{100}C_0\big(\frac{99}{100}\big)^{1/2}|\tilde{\finiteSetCurves}|^{1/2}\\
&\geq C_0|\tilde{\finiteSetCurves}|^{1/2}.
\end{split}
\end{equation}

Thus we can apply the induction hypothesis to $\tilde{\finiteSetCurves}$ (with $\tilde A = \frac{99}{100}A$) to conclude that there is a polynomial $P_2$ of degree
\begin{equation*}
\begin{split}
\deg P_2&\leq C_1 |\tilde{\finiteSetCurves}|/\tilde A\\
&\leq C_1 \frac{100}{99} \frac{101}{200} |\finiteSetCurves| / A\\
&\leq \frac{2}{3}C_1|\frac{|\finiteSetCurves|}{A}
\end{split}
\end{equation*}
that vanishes on $\tilde{\finiteSetCurves}$. Thus if we let $P=P_1P_2$, then $P$ vanishes on every curve of $\finiteSetCurves$, and
\begin{equation}
\begin{split}
\deg P&\leq 100 (100p|\finiteSetCurves|)^{1/2}+\frac{2}{3}C_1\frac{|\finiteSetCurves|}{A}\\
&\leq \big(10^4 C_2+\frac{2}{3}C_1\big) \frac{|\finiteSetCurves|}{A}.
\end{split}
\end{equation}
If we select $C_1$ sufficiently large depending on $C_2$ (recall that $C_2$ is a sufficiently large absolute constant), then $(10^4 C_2+\frac{2}{3}C_1)\leq C_1$, and this completes the induction.
\end{proof}
\bibliographystyle{amsplain}

\begin{thebibliography}{99}
\vskip.125in

\bibitem{ChardinSocias}
M.~Chardin, G.~Moreno-Soc\'ias.
Regularity of lex-segment ideals: some closed formulas and applications.
\emph{Proc. AMS.} 131(4):1093--1102. 2002.
%
%
% \bibitem{cox}
% D. Cox, J. Little and D. O'Shea,
% {\it Ideals, Varieties, and Algorithms: An Introduction to
% Computational Algebraic Geometry and Commutative Algebra}, 3rd edition,
% Springer-Verlag, Heidelberg, 2007.
%
%
\bibitem{Fulton} W.~Fulton. \emph{Intersection theory}. Ergebnisse der Mathematik und ihrer Grenzgebiete. 3. Folge. Springer-Verlag, Berlin, second edition, 1998.
%
%
%
\bibitem{GM}M.~Green, I.~Morrison. The equations defining Chow varieties. \emph{Duke Math. J.} 53:733--747. 1986.
%
%
\bibitem{Guth}L.~Guth. Distinct distance estimates and low degree polynomial partitioning. \emph{Discr. Comput. Geom}. 53(2):428--444. 2015.
%
%
\bibitem{GuthKatz2}L.~Guth, N.~Katz. Algebraic Methods in Discrete Analogs of the Kakeya Problem. \emph{Adv. Math.}. 225(5): 2828--2839. 2010. Also in arXiv:0812.1043.
%
%

\bibitem{GuthKatz}---
On the Erd\H{o}s distinct distance problem in the plane. \emph{Ann of Math.} 181(1): 155--190. 2015. Also in arXiv:1011.4105.	
%
%
\bibitem{harris} J.~Harris. \emph{Algebraic geometry: a first course}. Springer, New York, NY. 1995.
%
%
\bibitem{Katz}N.~Katz. The flecnode polynomial: a central object in incidence geometry. arXiv:1404.3412. 2014.
%
%
\bibitem{Kollar}J.~Kollar. Szemer\'edi-Trotter-type theorems in dimension 3. arXiv:1405.2243. 2014.
%
%
\bibitem{Landsberg}J.~Landsberg. \emph{Cartan for Beginners: Differential Geometry Via Moving Frames and Exterior Differential Systems.} American Mathematica Society. 2003.
%
%
\bibitem{matsumura}H.~Matsumura. \emph{Commutative Algebra}, second edition.  W. A. Benjamin Co., New York NY. 1980.
%
%
%\bibitem{encyclopedia}(http://www.encyclopediaofmath.org/index.php/Chow\_variety)
%
%
\bibitem{Milne}J.S.~Milne. Algebraic Geometry. (v6.00). Available at www.jmilne.org/math.  2014.
%
%
\bibitem{Salmon} G. Salmon. \emph{A Treatise on the Analytic Geometry of Three Dimensions}, Vol. 2, 5th
edition. Hodges, Figgis And Co. Ltd. 1915.
%
%
\bibitem{Schenck}H.~Schenck. \emph{Computational Algebraic Geometry}. Cambridge University Press. 2003.
%
%
\bibitem{SS}M.~Sharir, N.~Solomon. Incidences between points and lines in three dimensions. arXiv:1501.02544. 2015.
\end{thebibliography}

\end{document}